\documentclass[reqno,11pt]{amsart}

\usepackage[top=2.0cm,bottom=2.0cm,left=3cm,right=3cm]{geometry}
\usepackage{amsthm,amsmath,amssymb,dsfont}
\usepackage{mathrsfs,amsfonts,functan,extarrows,mathtools}
\usepackage[colorlinks]{hyperref}

\usepackage{marginnote}
\usepackage{xcolor}
\usepackage{stmaryrd,bm}
\usepackage{esint}
\usepackage{graphicx}
\usepackage[english]{babel}

\newtheorem{theorem}{Theorem}[section]
\newtheorem{proposition}{Proposition}[section]

\newtheorem{lemma}{Lemma}[section]

\numberwithin{equation}{section}
\allowdisplaybreaks

\def\d{\mathrm{d}}
\def\no{\nonumber}
\def\R{\mathbb{R}}
\def\eps{\epsilon}

\def\J{\mathcal{J}}

\newcounter{wronumber}

\begin{document}
\title
			{GLOBAL CLASSICAL SOLUTIONS TO AN EVOLUTIONARY MODEL FOR MAGNETOELASTICITY}

\author[N. Jiang]{Ning Jiang}
\address[Ning Jiang]{\newline School of Mathematics and Statistics, Wuhan University, Wuhan, 430072, P. R. China}
\email{njiang@whu.edu.cn}

\author[H. Liu]{Hui Liu}
\address[Hui Liu]
{\newline School of Mathematics and Statistics, Wuhan University, Wuhan, 430072, P. R. China}
\email{hui.liu@whu.edu.cn}

\author[Y.-L. Luo]{Yi-Long Luo}
\address[Yi-Long Luo]
		{\newline School of Mathematics and Statistics, Wuhan University, Wuhan, 430072, P. R. China}
\email{yl-luo@whu.edu.cn}

\maketitle

\begin{abstract}
  In this paper, we first prove the local-in-time existence of the evolutionary model for magnetoelasticity with finite initial energy by employing the nonlinear iterative approach given in \cite{Jiang-Luo-2019-SIAM} to deal with the geometric constraint $M \in \mathbb{S}^{d-1}$ in the Landau-Lifshitz-Gilbert (LLG) equation. Inspired by \cite{Lin-Liu-Zhang-CPAM2005, Lin-Zhang-2008-CPAM}, we reformulate the evolutionary model for magnetoelasticity with vanishing external magnetic field $H_{ext}$, so that a further dissipative term will be sought from the elastic stress. We thereby justify the global well-posedness to the evolutionary model for magnetoelasticity with zero external magnetic field under small size of initial data.\\


\end{abstract}



\section{Introduction}\label{Sec-Intro}

%
%

\subsection{Modeling of magnetoelasticity}
The discovery of magnetoelasticity dates back at least to the 19th century (see \cite{Brown-1966}). Magnetoelastic interaction or simply magnetoelasticity describes a class of phenomena on the interaction between elastic and magnetic effects: if a ferromagnetic rod is subject to a magnetizing field, the rod changes not only its magnetization but also its length, and in the opposite way, if the rod experiences tension, its length as well as its magnetization changes. Modeling of magnetoelastic materials goes back to Brown \cite{Brown-1966} as well as Tiersten \cite{Tiersten-1964-JMP, Tiersten-1965-JMP}. Regarding the analytical works,  besides many studies on static case relying on energy minimization \cite{DeSimone-Dolzmans-ARMA1998, DeSimone-James-2002, James-Kinderlehrer-1993}, rate-independent evolution models were investigated in
\cite{Kruzik-Stefanelli-Zeman} using the concept of energetic solutions, see also \cite{Mielke-Roubicek}. In micromagnetics, the dynamics is usually governed by the Landau-Lifshitz-Gilbert (LLG) equations \cite{Gilbert-1955, Gilbert-2004, LL-1935}, which has been extensively studied analytically, see for example \cite{A-S-NA1992, CF-2001, Melcher-2007, Melcher-2010}. If the LLG equation is coupled with elasticity, there are some works in the small strain setting \cite{CEF-2011, CSVC-2009}.

The major difficulty in analyzing magneoelastic models lies in the fact that usually elasticity is formulated in the reference configuration, while micromagneticsis modeled in the deformed configuration. This difficulty might disappear in the aforementioned small-strain setting in which the difference between the actual and reference configuration is neglected \cite{CEF-2011, CSVC-2009}. In general setting, one might transform the magnetic part back into the reference configuration as in \cite{DeSimone-Dolzmans-ARMA1998, DeSimone-James-2002, Kruzik-Stefanelli-Zeman}. However, this is only possible if one can assure, by suitable modeling assumptions, that the deformation is invertible. In the static case, this can be enforced by suitable coercivity of the elastic energy, in particular, the energy has to blow up as the determinant of the deformation gradient tends to zero. The dynamic case is more involved because the balance law for the deformation also features the inertia term. On the other hand, in \cite{RT-2017}, a magnetoelastic model is formulated in the fully Lagrangian setting. However, mathematical analysis of such a model could only be performed under several simplifying assumptions.

Recently, an energetic variational approach is taken to formulate the fully nonlinear problem of magnetoelsticity completely in Eulerian coordinates in the current configuration \cite{BFLS-2018-SIAM, Forster-2016}. By employing the idea of \cite{Liu-Walkington-ARMA2001}, a transport equation for the deformation gradient is found to allow one to obtain the deformation gradient in the current configuration from the velocity gradient. Consequently, the major obstacle from the point view of elasticity, the invertibility of the deformation is resolved. As for the magnetic part, the evolution of magnetization is modeled by Landau-Lifshitz-Gilbert equation with the time derivative replaced the convective one, which is in order to take into account that changes of the magnetization also occur due to transport by underlying viscoelastic material. By this approach, the evolution of magnetoelasticity is modeled by :
\begin{equation}\label{MEL-original}
\left\{
\begin{array}{l}
\partial_t v + v \cdot \nabla v + \nabla p + \nabla \cdot (2A \nabla M \odot \nabla M - W^{\prime}(F) F^{\top})- \nu \Delta v = \ \mu_{0} (\nabla H_{ext})^{\top} M  \,, \\
\nabla \cdot v = 0 \,, \\
\partial_t F + v \cdot \nabla F- \nabla v F= \kappa \Delta F \,, \\
\partial_t M + v \cdot \nabla M  = - \gamma M \times (2A \Delta M + \mu_{0} H_{ext})- \lambda M \times [M \times (2A \Delta M + \mu_{0} H_{ext} )] \,,
\end{array}
\right.
\end{equation}
in $\mathbb{R}^+ \times \mathbb{R}^d$ for $d = 2, 3$, which describes the magnetoelastic material responding to applied magnetic fields and reacting with a change of magnetization to mechanical stress.

 In the system \eqref{MEL-original}, the first equation of the bulk velocity $v (t,x) \in \mathbb{R}^d$ is the balance of momentum in Eulerian coordinates with stress tensor $\mathcal{T} = - p I + \nu ( \nabla v + \nabla v^\top ) + W' (F) F^\top - 2 A \nabla M \odot \nabla M$, where $p (t,x) \in \mathbb{R}$ is the hydrodynamic pressure, $W' (F) F^\top - 2 A \nabla M \odot \nabla M$ is the magnetoelastic part of the stress with $(\nabla M \odot \nabla M)_{ij} = \sum_{k} \partial_i M_k \partial_j M_k$ and $W(F)$ is the elastic energy. From now on, we will use the Einstein summation convection, for example, $(\nabla M \odot \nabla M)_{ij} = \partial_i M_k \partial_j M_k$. $F(t,x) = (F^{ij} (t,x))_{1 \leq i,j \leq d} \in \mathbb{R}^{d \times d}$ represents the deformation gradient with respect to the velocity $v (t,x)$ which obeys the evolution of the third evolution of deformation gradient in \eqref{MEL-original}, where $F^{ij}(t,x)$ is the entries of the $i$-th row and the $j$-th column of $F (t,x)$. We also denote by $(\nabla v)_{ij} = \partial_j v^i$. The last equation of \eqref{MEL-original} is the Landau-Lifshitz-Gilbert (LLG) equation with the effective magnetic field $ 2 A \Delta M + \mu_0 H_{ext} $, where $M(t,x) \in \mathbb{S}^{d-1}$ stands for the magnetization and $H_{ext} (t,x) \in \mathbb{R}^d$ denotes the given external magnetic field.

 In this system, $\nu > 0$ is the viscosity of the fluid, $\gamma > 0$ is the electron gyromagnetic ratio, $\lambda > 0$ is a phenomenological damping parameter, $A , \mu_0 > 0$ are the parameters coming from the Helmholtz free energy, and the constant $\kappa \geq 0$ as shown in \cite{BFLS-2018-SIAM}. We emphasize that $\kappa = 0$ is physically reasonable, which ensures that $\det F = 1$ is equivalent to the incompressibility $\nabla \cdot v = 0$ if we initially assume $\det F_0 = 1$, while $\kappa > 0$ is just a regularization of the deformation gradient $F$, which cannot make sure that $\det F = 1$ under the assumption of incompressibility. In this sense, we are more interested in the case $\kappa = 0$. Since constants are irrelevant for mathematical analysis, here and in the following we set $A = \tfrac{1}{2}$, $\mu_0 = 1$, $\gamma = \lambda = 1$ and take the elastic stress $W(F) = \tfrac{1}{2} |F|^2$. More precisely, we will consider the following simplified system in current paper:
\begin{equation}\label{MEL}
  \left\{
    \begin{array}{l}
      \partial_t v + v \cdot \nabla v + \nabla p + \nabla \cdot ( \nabla M \odot \nabla M - F F^\top ) - \nu \Delta v = ( \nabla H_{ext} )^\top M \,, \\
      \nabla \cdot v = 0 \,, \\
      \partial_t F + v \cdot \nabla F = \nabla v F \,, \\
      \partial_t M + v \cdot \nabla M = \Delta M + H_{ext} + \Gamma (M) M - M \times ( \Delta M + H_{ext} ) \,, \\
      |M| = 1 \,,
    \end{array}
  \right.
\end{equation}
where $\Gamma (M)$ is the Lagrangian multiplier
\begin{equation}\label{Lagrangian-Multiplier}
  \begin{aligned}
    \Gamma (M) = |\nabla M|^2 - M \cdot H_{ext} \,.
  \end{aligned}
\end{equation}
Here the last $M$-equation is equivalent to the LLG equation in \eqref{MEL-original} provided that $|M| = 1$. Moreover, the initial data of \eqref{MEL} is imposed on
\begin{equation}\label{IC-MEL}
\begin{aligned}
v(0, x) = v_{0}(x) \in \mathbb{R}^d \,,\ F(0, x) = F_{0}(x) \in \mathbb{R}^{d \times d} \,,\  M(0, x) = M_{0}(x) \in \mathbb{S}^{d-1}
\end{aligned}
\end{equation}
with compatibilities $\nabla \cdot v_{0}(x) = 0$ and $\det F_0 (x) = 1$.

\subsection{Analytical difficulties and ideas}
The system \eqref{MEL} can be viewed as a nonlinear coupling of hydrodynamics of viscoelasticity and LLG, each of which the analytical studies have been extensive in the past two decades. However, for the coupling, i.e. the system \eqref{MEL}, the research on the well-posedness is very few.  In \cite{BFLS-2018-SIAM}, global in time weak solutions are constructed by using Galerkin method and a fixed point argument. The proof in \cite{BFLS-2018-SIAM} combines the ideas of Lin-Zhang \cite{Lin-Liu-CPAM1995} on the liquid crystal flow and Carbou-Fabrie \cite{CF-2001} on the Landau-Lifshitz equation. See also the recent progress in this direction \cite{KKS-2019} and \cite{Zhao-DCDS2019}. We emphasize that the so-called weak solutions in \cite{BFLS-2018-SIAM} and \cite{KKS-2019} are not in the usual sense: the spatial regularity requirement on $M$ is not $H^1$, but higher ($H^3$ in \cite{BFLS-2018-SIAM} and $H^2$ in \cite{KKS-2019}). Furthermore, both \cite{BFLS-2018-SIAM} and \cite{KKS-2019} consider the regularized transport equation version, i.e. the system \eqref{MEL-original}, but not the physical case $\kappa = 0$, i.e. the system \eqref{MEL} considered in the current paper. In addition, the size of the initial deviation of $M$ to the identity matrix $I$ are required to be small. All of these unsatisfactory assumptions are coursed by the essential difficulties of the lack of regularity of the transport equation and the geometric constraint of $M$, $|M|=1$, which is extremely hard to be approximated in weaker norm. In fact, it is still highly nontrivial even in higher regular norm considered in this paper. We remark that in \cite{Zhao-DCDS2019}, a local well-posedness and blow-up criteria of classical solutions are established for the modified \eqref{MEL}, i.e. replacing the constraint $|M|=1$ by the usual Ginzburg-Landau approximation.

In this paper, instead of Ginzburg-Landau approximation or regularized version \eqref{MEL-original} with $\kappa > 0$, we consider the physical model \eqref{MEL} with $\kappa =0$ and with the geometric constraint $|M|=1$. Furthermore, we work in the context of classical solutions, rather than the weak-strong solution studied in \cite{BFLS-2018-SIAM} and \cite{KKS-2019}. Our aim is to provide a first existence theory of global classical solutions to the full evolutionary magnetoelastic model \eqref{MEL}. We state the main difficulties and the ideas of our methods in the rest of this subsection.

The system \eqref{MEL} can be regarded as an incompressible viscoelastic fluid system of $(v,F)$ in the Oldroyd model coupled with the LLG equation which describes the micromagnetics. The LLG equation is a type of heat flow to the unit sphere, which is closely related to the liquid crystal model. The $F(t,x)$ is the deformation gradient tensor with respect to the velocity field $v (t,x)$. More precisely, we define the following ordinary differential equation:
\begin{equation}\label{Trajectory}
  \begin{aligned}
    \partial_t x (t, X) = v ( t, x (t, X) ) \,, \quad x (0,X) = X \,.
  \end{aligned}
\end{equation}
The deformation tensor $\bar{F} (t,X)$ is then defined as
\begin{equation*}
  \bar{F} (t, X) : = \nabla_X x (t,X) \,.
\end{equation*}
Then the deformation tensor $F(t,x)$ is
$$ F(t,x) = \bar{F} (t, X^{-1} (t,x)) \,, $$
where $X^{-1} (t,x)$ is the inverse mapping of $x (t, X)$ defined by \eqref{Trajectory} and $F$ will automatically satisfy $\partial_t F + v \cdot \nabla F = \nabla v F$. Furthermore, the incompressibility of the fluid can then be represented as
\begin{equation}\label{Incomprsblt-F}
  \begin{aligned}
    \det F = 1 \,,
  \end{aligned}
\end{equation}
which is equivalent to $\nabla \cdot v = 0$ if we assume $\det F_0 = 1$.

The first goal of this paper is to prove the local existence result to the system \eqref{MEL} with initial data \eqref{IC-MEL}. We point out that the cancelations (the bracket $\langle \cdot \,, \cdot \rangle$ denotes the standard $L^2$ inner product):
\begin{equation}\label{Cancellation-M}
  \begin{aligned}
    \big\langle \nabla \cdot ( \nabla M \odot \nabla M ) , v \big\rangle - \big\langle v \cdot \nabla M , \Delta M \big\rangle = 0
  \end{aligned}
\end{equation}
and
\begin{equation}\label{Cancellation-F}
  \begin{aligned}
    \big\langle \nabla \cdot ( F F^\top ) , v \big\rangle + \big\langle \nabla v F , F \big\rangle = 0
  \end{aligned}
\end{equation}
play an essential role in the derivation of the a priori estimates to the system \eqref{MEL}. However, as in Ericksen-Leslie's liquid crystal model, one of the difficulty is to deal with the geometric constraint $|M| = 1$, which is highly nonlinear in particular for the norm with higher derivatives. Inspired by the work \cite{Jiang-Luo-2019-SIAM} of the first and the third authors of current paper, the key point is that if $\Gamma (M)$ in the last LLG equation of \eqref{MEL} is of the form \eqref{Lagrangian-Multiplier} and the initial data $M_0$ satisfies $|M_0| = 1$, then the constraint $|M| = 1$ for the {\em solution} to \eqref{MEL} will be {\em forced} to hold at any time $t \geq 0$. This means that $|M| = 1$ need only be given on the initial condition, while in the system \eqref{MEL}, we do not require $|M| = 1$ explicitly. These facts will be rigorously verified in Lemma \ref{Lmm-|M|=1}.

The other goal is to construct the global existence result to the system \eqref{MEL} with vanishing external magnetic field, i.e., $H_{ext} \equiv 0$ under small size of the initial data. We emphasize that because of the non-vanishing of external magnetic field $H_{ext}$, the energy of whole system does not decay. So we assume $H_{ext} = 0$ in proving the global existence result. The main difficulty is the absence of apparent dissipative mechanism in the evolution of the deformation gradient $F$. We employ the ideas in Lin-Liu-Zhang's work \cite{Lin-Liu-Zhang-CPAM2005, Lin-Zhang-2008-CPAM}. More precisely, we introduce
\begin{equation}
  \begin{aligned}
     U = F^{-1} \,, \qquad G = U - I \,,
  \end{aligned}
\end{equation}
where $I$ is the $d \times d$ unit matrix and $U$ obeys the evolution
\begin{equation}
  \begin{aligned}
    \partial_t U + v \cdot \nabla U + U \nabla v = 0 \,.
  \end{aligned}
\end{equation}
Moreover, the matrix-valued function $G(t,x)$ is of an important property (see (1.8) in \cite{Lin-Zhang-2008-CPAM})
\begin{equation}\label{curl-free-G}
  \begin{aligned}
    \partial_i G^{jk} = \partial_k G^{ji} \qquad \forall \ i,j,k = 1,2,\cdots , d \,,
  \end{aligned}
\end{equation}
which means that the matrix $G$ is curl free. Consequently, there exists a $\mathbb{R}^d$-valued function $\psi (t,x) = ( \psi^1 (t,x), \psi^2 (t,x) , \cdots , \psi^d (t,x) )$ such that
\begin{equation}\label{Gradient-psi}
  \begin{aligned}
    ( G^{j1}, G^{j2}, \cdots, G^{jd} ) = \nabla \psi^j \qquad \forall \ j = 1,2, \cdots, d \,.
  \end{aligned}
\end{equation}
Thanks to (3.3) of \cite{Lin-Zhang-2008-CPAM}, we know that
\begin{equation}\label{G-F-relation}
  \begin{aligned}
    \| G (t) \|_{H^s} = \| F^{-1} (t) - I \|_{H^s} \thicksim \| F(t) - I \|_{H^s} \,,
  \end{aligned}
\end{equation}
where the Hilbert space $H^s$ will be defined precisely latter. For the elastic stress $F F^\top$, by using Taylor's expansion
\begin{equation}
  \begin{aligned}
    F F^\top = ( I + G )^{-1} (I + G)^{- \top} = I - G - G^\top + g (G)
  \end{aligned}
\end{equation}
with $g (G) = O (|G|^2)$. By employing the curl free property \eqref{curl-free-G}, we have $\nabla \cdot G^\top = \nabla \mathrm{tr} \, G$. As a consequence, we can reformulate \eqref{MEL} with $H_{ext} \equiv 0$ as
\begin{equation}\label{MEL-reformulate}
  \left\{
    \begin{array}{l}
      \partial_t v - \nu \Delta v + \Delta \psi + \nabla q = - v \cdot \nabla v + \nabla \cdot g (G) - \nabla \cdot ( \nabla M \odot \nabla M ) \,, \\
      \nabla \cdot v = 0 \,, \\
      \partial_t \psi + v = - v \cdot \nabla \psi \,, \\
      \partial_t M + v \cdot \nabla M = \Delta M + |\nabla M|^2 M - M \times \Delta M \,,\\
      |M| = 1 \,,
    \end{array}
  \right.
\end{equation}
where $q = p + \mathrm{tr} \, G$. As shown in Section 1.2 of \cite{Lin-Zhang-2008-CPAM}, we easily know that the $(v, \psi, M)$-system \eqref{MEL-reformulate} is equivalent to \eqref{MEL} with vanishing external magnetic field $H_{ext}$. We point out that the elastic stress $F F^\top$ will lead to a damping term $\Delta \psi$ under the relation \eqref{Gradient-psi}. To be more specified, the evolution of $\psi$ in \eqref{MEL-reformulate} reduces to
\begin{equation}
  \begin{aligned}
    \partial_t \Delta \psi + \tfrac{1}{\nu} \Delta \psi + \tfrac{1}{\nu} \Delta ( \nu v - \psi ) = - \Delta ( v \cdot \nabla \psi ) \,,
  \end{aligned}
\end{equation}
where $\tfrac{1}{\nu} \Delta \psi$ is just what we required. However, there is a new quantity $ \tfrac{1}{\nu} \Delta ( \nu v - \psi ) $ to be controlled. Fortunately, by the evolution of $v$ in \eqref{MEL-reformulate}, $w = \nu v - \psi $ subjects to a generalized Stokes system
\begin{equation}
  \left\{
    \begin{array}{l}
      - \Delta w + \nabla q = - \partial_t v - v \cdot \nabla v + \nabla \cdot g (G) - \nabla \cdot ( \nabla M \odot \nabla M ) \,, \\
      \nabla \cdot w = - \nabla \cdot \psi \,.
    \end{array}
  \right.
\end{equation}
Then Lemma \ref{Lmm-General-Stokes} will help us to dominate the quantity $ \tfrac{1}{\nu} \Delta ( \nu v - \psi ) $. What we want to emphasize is that the incompressibility \eqref{Incomprsblt-F} reduces to $\nabla \cdot \psi = O ( |G|^2 )$, which plays an essential role in deriving the global energy estimate of \eqref{MEL-reformulate}. For details, see Section \ref{Sec-Global}. As a consequence, we can establish the global existence of small solutions to the Cauchy problem \eqref{MEL}-\eqref{IC-MEL} with the vanishing external magnetic field $H_{ext}$.

\subsection{Notations and main results} Before presenting the main results in this paper, we first introduce the notations used throughout the current work. For convention, we denote the usual $L^p (\mathbb{R}^d) $ space by $L^p$ for $p \in [ 1 , \infty ]$, where $d = 2,3$. More precisely, $\| f \|_{L^p} = \big( \int_{\mathbb{R}^d} |f|^p \d x \big)^\frac{1}{p}$ if $p \in [ 1 , \infty )$ and $\| f \|_{L^\infty} = \mathrm{ess \ sup}_{x \in \mathbb{R}^d} |f (x)|$. For $p = 2$, we use the notations $\langle \cdot , \cdot \rangle $ to represent the inner product on the Hilbert space $L^2$.

In this paper, $\nabla$ stands for the gradient operator, $\nabla \cdot$ denotes the divergence operator and $\partial_i$ means $\partial_{x_i}$ for all $i = 1, 2, \cdots, d$. The symbol $\Delta = \partial_i \partial_i$ is the Laplacian operator. For any matrix-valued function $G = G^{ij}$, $(\nabla \cdot G)_i = \partial_j G^{ji}$ for all $1 \leq i \leq d$. For any multi-indexes $m = (m_1, m_2, \cdots, m_d) \in \mathbb{N}^d$, we denote the $m^{th}$ partial derivative by
\begin{equation*}
  \begin{aligned}
    \partial^m = \partial^{m_1}_{x_1} \partial^{m_2}_{x_2} \cdots \partial^{m_d}_{x_d} \,.
  \end{aligned}
\end{equation*}
If each component of $m \in \mathbb{N}^d$ is not greater that that of $\widetilde{m}$'s, we denote by $m \leq \widetilde{m}$. The symbol $m < \widetilde{m}$ means $m \leq \widetilde{m}$ and $|m| < |\widetilde{m}|$, where $|m| = m_1 + m_2 + \cdots + m_d$. We denote the Sobolev space $W^{s,p} = W^{s,p} (\mathbb{R}^d)$ by the norm
\begin{equation*}
  \begin{aligned}
    \| f \|_{W^{s, p}} = \Big( \sum_{|m| \leq s} \int_{\mathbb{R}^d} |\partial^m f (x)|^p \d x \Big)^\frac{1}{p} \,.
  \end{aligned}
\end{equation*}
If $p = 2$, we denote by the Hilbert space $H^s = H^s (\mathbb{R}^d) : = W^{s,2}$.

In this paper, we will prove the local well-posedness for the evolution model for the magnetoelastic system \eqref{MEL} with initial data \eqref{IC-MEL}. Moreover, we will justify the global existence under small size of initial data with vanishing external magnetic field $H_{ext}$. Now we precisely state our main theorems as follows:
\begin{theorem}[Local well-posedness]\label{Thm-1}
Let the integer $s \geq 2$ and $d = 2$ or $3$. Given the external magnetic field $H_{ext} (t,x) \in L^\infty (\mathbb{R}^+ ; H^s (\mathbb{R}^d))$ and the initial data $(v_0 (x), F_0 (x) , M_0 (x))\in \mathbb{R}^d \times \mathbb{R}^{d \times d} \times \mathbb{S}^{d-1}$ satisfying $ v_0, F_0 , \nabla M_0 \in H^s (\mathbb{R}^d) $, if the initial energy
$$\mathcal{E}_0 : = \| v_0 \|_{H^s}^2 + \| F_0 \|^2_{H^s} + \| \nabla M_0 \|^2_{H^s} < \infty \,,$$
then there exists $T > 0$, depending only on $ \mathcal{E}_0, H_{ext}$, such that the Cauchy problem \eqref{MEL}-\eqref{IC-MEL} admits a unique solution $(v,F,M)$ satisfying
\begin{equation*}
  \begin{aligned}
    v \in L^\infty(0,T;H^s(\mathbb{R}^d)) \cap L^2(0,T;H^{s+1}(\mathbb{R}^d)) \,, F \in L^\infty(0,T;H^s(\mathbb{R}^d)) \,, \nabla M \in L^\infty(0,T;H^s(\mathbb{R}^d)) \,.
  \end{aligned}
\end{equation*}
Moreover, the solution $(v,F,M)$ obeys the energy bound
\begin{equation*}
  \begin{aligned}
    \sup_{t \in [0,T]} & \big{(} \| v \|^2_{H^s} +  \| F \|^2_{H^s} + \| \nabla M \|^2_{H^s} \big{)} + \int_0^T \big{(} \nu \| \nabla v \|^2_{H^s} + \| \Delta M \|^2_{H^s}  \big{)} \d t \leq C \,,
  \end{aligned}
\end{equation*}
where the positive constant C depends only on $\mathcal{E}_0$, $H_{ext}$ and $T$.
\end{theorem}

\begin{theorem}[Global well-posedness]\label{Thm-Global}
	Under the assumptions of Theorem \ref{Thm-1}, we assume further that $H_{ext} = 0$ and $G_0 = F_0^{-1} - I$ satisfies \eqref{curl-free-G} and
	\begin{equation}\label{IC-small-size}
	  \begin{aligned}
	    \det F_0 = 1 \,, \quad \| v_0 \|^2_{H^s} + \| F_0 - I \|^2_{H^s} + \| \nabla M_0 \|^2_{H^s} \leq \eps_0
	  \end{aligned}
	\end{equation}
	for some $\eps_0 > 0$ sufficiently small. Then the solution $(v,F,M)$ constructed in Theorem \ref{Thm-1} is global with $ v \in L^\infty (\mathbb{R}^+; H^s (\mathbb{R}^d)) $, $\nabla v \in L^2 (\mathbb{R}^+; H^s(\mathbb{R}^d))$, $G, \nabla M \in L^\infty (\mathbb{R}^+;H^s(\mathbb{R}^d))$ and $\Delta M \in L^2 ( \mathbb{R}^+ ; H^s(\mathbb{R}^d) )$, where $G = F^{-1} - I$. Moreover, there is a $C > 0$ such that
	\begin{equation}
	  \begin{aligned}
	    \| v (t) \|^2_{H^s} + \| \nabla M (t) \|^2_{H^s} + \| G (t) \|^2_{H^s} + \| \partial_t v (t) \|^2_{H^{s-2}} + \| \partial_t G (t) \|^2_{H^{s-2}} \\
	    + \int_0^t [ \nu \| \nabla v (\tau) \|^2_{H^s} + \| \Delta M ( \tau ) \|^2_{H^s} ] \d \tau \leq C \eps_0
	  \end{aligned}
	\end{equation}
	for all $t \geq 0$.
\end{theorem}

We now sketch the main ideas of the proof of the above theorems. When proving the local well-posedness result with large initial data, the geometric constraint $|M|=1$ will be difficult to be held in the construction of the approximate system. However, inspired by \cite{Jiang-Luo-2019-SIAM}, this constraint should be set initially and then will be satisfied during the evolution if the solutions are smooth. More precisely, we consider $h(x,t)=\mid M(x,t) \mid^{2}-1$. It satisfies a nonlinear parabolic equation with a given velocity field $v$ under the vanishing initial data $h(0,x) = 0$, which implies that $h(t,x) = 0$ holds for all $t$ by employing the Gr\"onwall inequality. So, we first construct the local solutions to the LLC equation with a given background velocity $v$ by the mollifying approximate, refer to \eqref{ch5-3}. In the mollifying approximate system \eqref{ch5-3}, the constraint $M^\eps \in \mathbb{S}^{d-1}$ does not hold as well as $\mathcal{J}_\eps M_0$. When we derive the uniform energy bound to \eqref{ch5-3}, we need to control the quantity $\| M^\eps \|_{L^\infty}$ by applying the equality \eqref{M-L-infty}, i.e.,
\begin{equation*}
  \begin{aligned}
    \| M^\eps \|_{L^\infty} \leq \| M^{\varepsilon}-\J_\eps M^{in}\|_{L  ^{2}}+C\|\nabla M^{\varepsilon}\|_{H^{N}}+C\| M^{in}\|_{H^{N}}+1 \,.
  \end{aligned}
\end{equation*}
We thereby naturally design the quantity $ \| M^\eps - \mathcal{J}_\eps M_0 \|_{L^2} $, which initially vanishes, as a part of the approximate energy. In fact, our approximate energy is $ E_\eps (t) = \| \nabla M^\eps \|^2_{H^s} + \| M^\eps - \mathcal{J}_\eps M_0 \|^2_{L^2} $ with $E_\eps (0) = \| \nabla \mathcal{J}_\eps M_0 \|^2_{H^s} \leq \| \nabla M_0 \|^2_{H^s}$. Then, combining Lemma \ref{Lmm-|M|=1}, we can derive the local existence result of the LLG equation with a given background velocity $v$.

Next we carefully design a nonlinear iteration scheme to construct the approximate solutions $(v^n, F^n, M^n)$, where the nonlinearity is due to keeping the constraint $| M^n | = 1$: solve the Stokes type equation of $v^{n+1}$ in terms of $v^n$, $F^n$ and $M^n$, solve the linear ODE type equation of $F^{n+1}$ in terms of $v^n$ and $F^n$, and solve the heat flow type equation of $M^{n+1}$ in terms of $v^n$ and $M^n$. We want emphasize that the key point of this construction is that the geometric constraint $ M^n \in \mathbb{S}^{d-1}$ will not cause any extra difficulty. We can thereby derive the uniform energy estimate by using $|M^n| = 1$, which will greatly simplify the process. Consequently, the existence of the local-in-times smooth solutions can be proved.

To prove the global small classical solution to \eqref{MEL} with $H_{ext}=0$, we only need to derive the global uniform energy bound of the reformulated system \eqref{MEL-reformulate}. In order to seek a dissipative mechanism of $\psi$, we are required to control the quantity $w = \nu v - \psi$ by employing Lemma \ref{Lmm-General-Stokes} and the key structure \eqref{Key-Structure}. However, there is an extra norm $\| \partial_t v \|_{H^{s-2}}$ is uncontrolled. Considering the coupling of $v$ and $\psi$ in \eqref{MEL-reformulate}, we have to additionally estimate the norms $\| \partial_t v \|_{H^{s-2}}$ and $\| \partial_t \psi \|_{H^{s-2}}$ as in \eqref{norm-vt-Hs-2} and \eqref{norm-nabla-psi-t-Hs-2-1} respectively.

The organization of current paper is as follow: in the next section, we provide an a priori estimate of the system \eqref{MEL}. In Section \ref{Sec: Local-LLG-Given-v}, we first justify the relation between the Lagrangian multiplier $\Gamma (M)$ and $|M|=1$, which guarantees that the condition of unit length of $M$ needs only be imposed initially. Then we show the local existence of LLG equation for $M$ with a given velocity $v$
, which will be applied to construct the iterative approximate equations of \eqref{MEL}-\eqref{IC-MEL}. In Section \ref{Sec: Local-Result}, we prove the local well-posedness of \eqref{MEL} with large initial data by deriving uniform energy bounds of the iterative approximate system \eqref{MEL-app}. In Section \ref{Sec-Global}, by employing the equivalent reformulation \eqref{MEL-reformulate} of \eqref{MEL}, we globally extend the solution of \eqref{MEL}-\eqref{IC-MEL} constructed in Section \ref{Sec: Local-Result} under the small initial energy condition \eqref{IC-small-size}.

\section{A priori estimate}\label{A_Priori}

In this section, the {\em a priori} estimate of the system \eqref{MEL} will be accurately derived from employing the energy method. There are two key cancellations \eqref{Cancellation-M} (between $v$ and $M$) and \eqref{Cancellation-F} (between $v$ and $F$) will play an essential role in deriving the energy estimate. These cancellations will present the forms
\begin{equation}\label{Cancellation-M-H}
  \begin{aligned}
    & \big\langle \partial^m \nabla \cdot ( \nabla M \odot \nabla M ) , \partial^m v \big\rangle - \big\langle \partial^m ( v \cdot \nabla M ) , \Delta \partial^m M \big\rangle \\
    = & \sum_{0 \neq m' \leq m} C_m^{m'} \Big[ \big\langle  \partial^m v \cdot \nabla \partial^{m'} M , \Delta \partial^{m-m'} M \big\rangle - \big\langle \partial^{m-m'} v \cdot \nabla \partial^{m'} M , \Delta \partial^m M \big\rangle \Big]
  \end{aligned}
\end{equation}
and
\begin{equation}\label{Cancellation-F-H}
  \begin{aligned}
    & \big\langle \partial^m \nabla \cdot ( F F^\top ) , \partial^m v \big\rangle + \big\langle \partial^m ( \nabla v F ) , \partial^m F \big\rangle \\
    = & \sum_{0 \neq m' \leq m} C_m^{m'} \Big[ \big\langle \nabla \partial^{m-m'} v \partial^{m'} F , \partial^m F \big\rangle - \big\langle \partial^{m-m'} F \partial^{m'} F^\top , \nabla \partial^m v \big\rangle \Big]
  \end{aligned}
\end{equation}
for all multi-indexes $m \in \mathbb{N}^d$ with $|m| \geq 1$, where the terms of right-hand side in \eqref{Cancellation-M-H} and \eqref{Cancellation-F-H} are lower order derivatives to be easily controlled. We now introduce the following energy functional $\mathcal{E}_s (t)$ and energy dissipative rate functional $\mathcal{D}_s (t)$:
\begin{equation*}
  \begin{aligned}
    \mathcal{E}_s (t) & := \| v \|^2_{H^s} +  \| F \|^2_{H^s} + \| \nabla M \|^2_{H^s} \,, \\
    \mathcal{D}_s (t) & := \nu \| \nabla v \|^2_{H^s} + \|\Delta M\| ^{2}_{H^{s}}   \,.
  \end{aligned}
\end{equation*}
Then, we articulate the following proposition:
\begin{proposition}[A priori estimate]\label{Prop-AE1}
  Let $s \geq 2$ be any fixed integer. Assume that $(v,F,M)$ is a sufficiently smooth solution to \eqref{MEL} on the interval $[0,T]$. Then there is a positive constant $C = C( \nu, s, d, H_{ext}) > 0$, such that
  \begin{equation}\label{Ener-Bnds-Local}
    \begin{aligned}
       &\tfrac{1}{2} \tfrac{\d}{\d t} \mathcal{E}_s (t) + \mathcal{D}_s (t) \leq C \big( \mathcal{E}^{\frac{1}{2}}_s (t) +\mathcal{E}^{\frac{3}{2}}_s (t) \big) \big( 1 + \mathcal{D}^{\frac{1}{2}}_s (t) \big)
     \end{aligned}
  \end{equation}
  holds for all $t \in [0,T]$.
\end{proposition}

\begin{proof}[Proof of Proposition \ref{Prop-AE1}]
We act $m$-order derivative on the first $v$-equation of \eqref{MEL} for all $|m| \leq s$, take $L^2$-inner product by dot with $\partial^{m}v $ and integrate by parts over $x \in \mathbb{R}^d$. We thereby have
  \begin{equation}\label{MEL_7}
     \begin{aligned}
       &\tfrac{1}{2} \tfrac{\d}{\d t} \|\partial^{m}v\|^{2}_{L^{2}}) + \langle \partial^{m}(v\cdot \nabla v),\partial^{m}v \rangle + \nu\|\nabla \partial^{m}v\|^{2}_{L^{2}} \\
       & + \langle \nabla \cdot \partial^{m}( \nabla M \odot \nabla M )- FF^{\top}),\partial^{m} v\rangle -\langle \partial^{m} (\nabla H_{ext})^{\top} M,\partial^{m}v \rangle =0\,.
    \end{aligned}
  \end{equation}
  Since $\nabla \cdot v = 0$, we straightforwardly calculate that
  \begin{equation*}
    \begin{aligned}
      \langle \partial^{m}(v\cdot \nabla v),\partial^{m}v \rangle=\langle v \cdot \nabla \partial^{m} v,\partial^{m}v \rangle+\sum_{\substack{0\neq m^{\prime} \leq m}}C^{m^{\prime}}_{m}\langle \partial^{m^{\prime}}v \cdot \nabla \partial^{m-m^{\prime}} v,\partial^{m}v \rangle\,,
    \end{aligned}
  \end{equation*}
  and
  \begin{equation*}
    \begin{aligned}
      &\langle \nabla \cdot \partial^{m}( \nabla M \odot \nabla M )- FF^{\top}),\partial^{m} v\rangle\\
      &=\langle \partial^{m}(\nabla\frac{|\nabla M|^{2}}{2}+\nabla M \cdot \Delta M ), \partial^{m} v\rangle+\langle \partial^{m}(FF^{\top}), \nabla \partial ^{m} v\rangle\\
      & =\langle \nabla M\cdot \Delta \partial^{m} M,  \partial^{m} v\rangle+\sum_{\substack{0\neq m^{\prime} \leq m}}C^{m^{\prime}}_{m}\langle \nabla \partial^{m^{\prime}} M \cdot \Delta \partial^{m-m^{\prime}} M,  \partial^{m} v\rangle\\
      &+ \langle \partial^{m}(FF^{\top}), \nabla \partial ^{m} v\rangle+\sum_{\substack{0\neq m^{\prime} \leq m}}C^{m^{\prime}}_{m} \langle \partial^{m-m^{\prime}}F \partial^{m^{\prime}}F^{\top}, \nabla \partial ^{m} v\rangle\,,
    \end{aligned}
  \end{equation*}
  and
  \begin{equation*}
    \begin{aligned}
      &-\langle \partial^{m} (\nabla H_{ext})^{\top} M,\partial^{m}v \rangle \\
      = & -\langle (\nabla \partial^{m} H_{ext})^{\top} M,\partial^{m}v \rangle-\sum_{\substack{0\neq m^{\prime} \leq m}}C^{m^{\prime}}_{m}\langle (\nabla \partial^{m-m^{\prime}} H_{ext})^{\top}\partial^{m^{\prime}} M,\partial^{m}v \rangle\,.
    \end{aligned}
  \end{equation*}
  Collecting the above equalities, the equality \eqref{MEL_7} can be rewritten as
  \begin{equation}\label{MEL_8}
    \begin{aligned}
       &\tfrac{1}{2} \tfrac{\d}{\d t} \|\partial^{m} v \|^{2}_{L^{2}} + \nu \|\nabla \partial^{m} v \|^{2}_{L^{2}} + \langle \nabla M \cdot \Delta \partial^{m} M,\partial^{m} v \rangle + \langle \partial^{m}F F^{\top},\nabla \partial^{m} v\rangle \\[2mm]
       = &\underset{A_0^{m}}{\underbrace{ \langle (\nabla \partial^{m} H_{ext})^{\top} M,\partial^{m} v \rangle}} \ \underset{A_1^{m}}{\underbrace{ - \sum_{\substack{0\neq m^{\prime} \leq m}} C^{m^{\prime}}_{m} \langle \partial^{m^{\prime}} v \cdot \nabla \partial^{m-m^{\prime}} v,\partial^{m} v \rangle}} \\
       & +\underset{A_2^{m}}{\underbrace{ \sum_{\substack{0\neq m^{\prime} \leq m}} C^{m^{\prime}}_{m} \langle \nabla \partial^{m^{\prime}} M \cdot \Delta \partial^{m-m^{\prime}} M,\partial^{m} v \rangle}} +\underset{A_3^{m}}{\underbrace{ \sum_{\substack{0\neq m^{\prime} \leq m}} C^{m^{\prime}}_{m} \langle \partial^{m-m^{\prime}} F \nabla \partial^{m^{\prime}} F^{\top},\nabla \partial^{m} v \rangle}}\\
       & -\underset{A_4^{m}}{\underbrace{ \sum_{\substack{0\neq m^{\prime} \leq m}} C^{m^{\prime}}_{m} \langle (\nabla \partial^{m-m^{\prime}} H_{ext})^{\top} \partial^{m^{\prime}} M, \partial^{m} v \rangle}} \,.
   \end{aligned}
 \end{equation}
From acting the $m$-order derivative on the second equation of the system \eqref{MEL} for all $|m| \leq s$, taking $L^2$-inner product by dot with $\partial^{m} F $ and integrating by parts over $x \in \mathbb{R}^d$, we deduce that
   \begin{equation*}
     \begin{aligned}
       & \tfrac{1}{2} \tfrac{\d}{\d t} \| \partial^{m} F \|^{2}_{L^{2}} + \langle \partial^{m} (v \cdot \nabla F), \partial^{m} F \rangle-\langle (\partial^{m} \nabla v F),\partial^{m} F \rangle =0 \,.
    \end{aligned}
  \end{equation*}
  By employing the divergence-free property of $v$ and the direct calculations, we have
   \begin{equation*}
     \begin{aligned}
       & \langle \partial^{m}(v\cdot \nabla F),\partial^{m} F \rangle = \sum_{\substack{0\neq m^{\prime} \leq m}} C^{m^{\prime}}_{m}\langle\partial^{m^{\prime}} v \cdot \nabla \partial^{m-m^{\prime}} F),\partial^{m} F \rangle\,,
    \end{aligned}
  \end{equation*}
  and
  \begin{equation*}
    \begin{aligned}
      -\langle \nabla \partial^{m} v F,\partial^{m}F \rangle=-\langle  \nabla \partial^{m} v F,\partial^{m}F \rangle-\sum_{\substack{0\neq m^{\prime} \leq m}} C^{m^{\prime}}_{m}\langle ( \nabla \partial^{m-m^{\prime}} v \partial^{m\prime}F),\partial^{m}F \rangle \,.
    \end{aligned}
  \end{equation*}
  As a consequence, we have
   \begin{equation}\label{MEL_9}
     \begin{aligned}
       \tfrac{1}{2} \tfrac{\d}{\d t} \|\partial^{m} F \|^{2}_{L^{2}} - \langle \nabla \partial^{m} v F, \partial^{m} F \rangle = & \ \underset{B_1^{m}}{\underbrace{ - \sum_{\substack{0 \neq m^{\prime} \leq m}} C^{m^{\prime}}_{m} \langle \partial^{m^{\prime}} v \cdot \nabla \partial^{m-m^{\prime}} F,\partial^{m} F \rangle}} \\
       & \underset{B_2^{m}}{\underbrace{ - \sum_{\substack{0 \neq m^{\prime} \leq m}} C^{m^{\prime}}_{m} \langle \nabla \partial^{m-m^{\prime}} v \cdot \partial^{m^{\prime}} F,\partial^{m} F \rangle}}\,.
    \end{aligned}
  \end{equation}
 We further act the $m$-order derivative on the last second $M$-equation of the system \eqref{MEL} for all $|m| \leq s$, take $L^2$-inner product by dot with $\Delta\partial^{m} M$ and integrate by parts over $x \in \mathbb{R}^d$. We thereby obtain
 \begin{equation}\label{MEL_10}
   \begin{aligned}
     &\tfrac{1}{2} \tfrac{\d}{\d t} \|\nabla \partial^{m} M\|^{2}_{L^{2}} + \|\Delta \partial^{m} M\|^{2}_{L^{2}} - \langle \partial^{m} v \cdot \nabla M, \Delta \partial^{m} M \rangle \\[2mm]
     = & \underset{C_1^{m}}{\underbrace{ \sum_{\substack{ 0 \neq m^{\prime} \leq m}} C^{m^{\prime}}_{m} \langle \partial^{m-m^{\prime}} v \nabla \partial^{m^{\prime}} M, \Delta \partial^{m} M \rangle}} \ \ \underset{C_2^{m}}{\underbrace{ - \sum_{\substack{ 0 \neq m^{\prime} \leq m}} C^{m^{\prime}}_{m} \langle \partial^{m} H_{ext},\Delta \partial^{m} M \rangle }} \\
     & \underset{C_3^{m}}{\underbrace{ - \sum_{\substack{ 0 \neq m^{\prime} \leq m}} C^{m^{\prime}}_{m} \langle \partial^{m^{\prime }} M \times \Delta\partial^{m-m^{\prime}} M ,\Delta \partial^{m} M \rangle}} \\
     & \underset{C_4^{m}}{\underbrace{ - \langle \partial^{m} ( M \times H_{ext}), \Delta \partial^{m}  M \rangle}}  + \underset{C_5^{m}}{\underbrace{ \langle \partial^{m} ( \Gamma(M) M), \Delta \partial^{m} M \rangle}} \,.
  \end{aligned}
\end{equation}
From adding \eqref{MEL_8}, \eqref{MEL_9} to \eqref{MEL_10} and summing up for all $|m| \leq s$, we derive that
\begin{equation}\label{MEL_11}
  \begin{aligned}
    &\tfrac{1}{2} \tfrac{\d}{\d t} \mathcal{E}_s (t) + \mathcal{D}_s (t) = \sum_{|m| \leq s} \Big( \sum_{0 \leq i \leq 4} A_i^m + \sum_{1 \leq i \leq 2} B_i^{m} + \sum_{1 \leq i \leq 5} C_i^{m} \Big) \,.
  \end{aligned}
\end{equation}

Now we estimate the terms $A_i^{m} (0 \leq i \leq 4), B_i^{m} (1 \leq i \leq 2), C_i^{m} (1 \leq i \leq 5) $ term by term. For the quantity $A_0^{m}$, we have
\begin{equation}\label{A0-m}
  \begin{aligned}
    A_{0}^{m} \leq & \|\nabla\partial^{m} H_{ext} \|_{L^{2}} \|\partial^{m} v\|_{L^{2}} \leq \|H_{ext}\|_{H^{s+1}} \|v\|_{H^{s}} \leq \|H_{ext}\|_{H^{s+1} } \mathcal{E}^{\frac{1}{2}}_s (t) \,.
  \end{aligned}
\end{equation}
Here we make use of the H\"older inequality. For the term $A_1^{m}$, we estimate that
\begin{equation}\label{A1-m}
  \begin{aligned}
    A_{1}^{m} & \leq C \sum_{\substack{0\neq m^{\prime} \leq m}} \|\partial^{m^{\prime}} v \|_{L^{4}}\|\nabla \partial^{m-m^{\prime}} v \|_{L^{4}}\|\partial^{m} v \|_{L^{2}}\\
    & \leq C \sum_{\substack{0\neq m^{\prime} \leq m}}\|\partial^{m^{\prime}} v \|_{H^{1}}\|\nabla \partial^{m-m^{\prime}} v \|_{H^{1}}\|\partial^{m} v \|_{L^{2}} \\
    & \leq C \| v \|^{2}_{H^{s}}\| \nabla v \|_{H^{s}} \leq C \mathcal{E}_s (t) \mathcal{D}^{\frac{1}{2}}_s (t) \,.
  \end{aligned}
\end{equation}
 For the terms $A_2^{m}, A_3^{m}, A_4^{m}$, the H\"older inequality, the Sobolev embedding $H^1(\R^d) \hookrightarrow L^p(\R^d)$ $(p=3,4)$ and $H^2(\R^d) \hookrightarrow L^\infty(\R^d)$ with $d = 2$ or $3$ enable us to estimate that
\begin{equation}\label{A2-m}
  \begin{aligned}
    A_{2}^{m} &  \leq C \sum_{\substack{ 0 \neq m^{\prime} < m}} \| \nabla \partial^{m^{\prime}} M \|_{L^4} \| \Delta \partial^{m-m^{\prime}} M \|_{L^4} \| \partial^{m} v \|_{L^{2}} + \| \nabla \partial^{m} M \|_{L^{2}} \| \Delta M\|_{L^{\infty}} \| \partial^{m} v\|_{L^{2}}\\
    & \leq C \| \nabla M \|^{2}_{H^{s}} \| v \|_{H^{s}} \leq C \mathcal{E}^{\frac{3}{2}}_s (t) \,,
  \end{aligned}
 \end{equation}
 where $s \geq 2$ is required, and
\begin{equation}\label{A3-m}
  \begin{aligned}
    A_{3}^m & \leq C \| F \|_{L^{\infty}} \| \partial^{m} F^{\top} \|_{L^{2}} \| \nabla \partial^{m} v \|_{L^{2}} + \sum_{\substack{ 0 \neq m^{\prime} < m}} \|\partial^{m-m^{\prime}} F\|_{L^{4}} \| \partial^{m} F^{\top} \|_{L^{4}} \| \nabla \partial^{m} v\|_{L^{2}} \\
    &\leq C \| F \|^{2}_{H^{s}} \|\nabla v \|_{H^{s}}  \leq C \mathcal{E}_s (t) \mathcal{D}^{\frac{1}{2}}_s (t) \,,
  \end{aligned}
\end{equation}
and
  \begin{equation}\label{A4-m}
    \begin{aligned}
      A_{4}^m & \leq C \sum_{\substack{|m^{\prime} | = 1}} \| \partial^{m-m^{\prime}} H_{ext} \|_{L^{2}} \| \partial^{m^{\prime}} M \|_{L^{\infty}} \| \partial^{m} v \|_{L^{2}} \\
      & + C \sum_{\substack{m^{\prime} \leq m, |m'| \geq 2}} \| \partial^{m-m^{\prime}} H_{ext} \|_{L^{4}} \| \partial^{m^{\prime}} M \|_{L^{4}} \|\partial^{m} v \|_{L^{2}} \\
      & \leq C \| H_{ext} \|_{H^{s}} \| \nabla M \|_{H^{s}} \| v \|_{H^{s}} \leq C \| H_{ext} \|_{H^{s}} \mathcal{E}_s (t) \,.
   \end{aligned}
 \end{equation}

 For the term $B_1^m $ and $ B_2^m$, from the H\"older inequality and the Sobolev embedding theory, we deduce that
 \begin{equation}\label{B1-m}
   \begin{aligned}
     B_{1}^m & \leq C \sum_{\substack{|m^{\prime} | = 1}} \| \partial^{m^{\prime}} v \|_{L^{\infty}} \| \nabla \partial^{m-m^{\prime}} F \|_{L^{2}} \|\partial^{m} F \|_{L^{2}} \\
     & + C \sum_{\substack{ 2 \neq m^{\prime} \leq m}} \| \partial^{m^{\prime}} v \|_{L^{4}} \|\nabla \partial^{m-m^{\prime}} F \|_{L^{4}} \| \partial^{m} F \|_{L^{2}} \\
     & \leq C \| \nabla  v \|_{H^{s}} \| F \|^{2}_{H^{s}} \leq C \mathcal{E}_s (t) \mathcal{D}^{\frac{1}{2}}_s (t) \,,
  \end{aligned}
\end{equation}
 and
 \begin{equation}\label{B2-m}
   \begin{aligned}
     B_{2}^m & \leq C \sum_{\substack{ 0 \neq m^{\prime}< m}} \| \nabla \partial^{m-m^{\prime}} v \|_{L^{4}} \| \partial^{m^{\prime}} F \|_{L^{4}} \|\partial^{m} F \|_{L^{2}} + C \|\nabla v \|_{L^{\infty}} \|\partial^{m} F \|^{2}_{L^{2}} \\
     & \leq \| \nabla v \|_{H^{s}} \| \partial^{m} F \|^{2}_{H^{s}} \leq C \mathcal{E}_s (t) \mathcal{D}^{\frac{1}{2}}_s (t) \,.
   \end{aligned}
 \end{equation}

 Next, we estimate the terms $C_i^m$ for $1 \leq i \leq 5$. The term $C_1^m$ can be estimated as
 \begin{equation}\label{C1-m}
   \begin{aligned}
     C_{1}^m & \leq C \| v \|_{L^{\infty}} \| \nabla \partial^{m} M \|_{L^{2}} \| \Delta \partial^{m} M \|_{L^{2}} \\
     & + C \sum_{\substack{ 0 \neq m^{\prime} < m}} \| \partial^{m-m^{\prime}} v \|_{L^{4}} \| \nabla \partial^{m^{\prime}} M \|_{L^{4}} \| \Delta \partial^{m} M \|_{L^{2}} \\
     & \leq C \| v \|_{H^{s}} \| \nabla M \|_{H^{s}} \| \Delta M \|_{H^{s}} \leq C \mathcal{E}_s (t) \mathcal{D}^{\frac{1}{2}}_s (t)\,,
   \end{aligned}
 \end{equation}
 where the H\"older inequality and the Sobolev embedding theory are utilized. Similarly in arguments of estimating the term $C_1^m$, one can easily estimate the terms $C_2^m$, $C_3^m$ and $C_4^m$ that
 \begin{equation}\label{C2-m}
   \begin{aligned}
     C_{2}^m & \leq \langle \nabla \partial^{m} H_{ext}, \nabla \partial^{m} M \rangle \leq \| \nabla \partial^{m} H_{ext} \|_{L^{2}} \| \nabla \partial^{m} M \|_{L^{2}}\\
     & \leq \| H_{ext} \|_{H^{s+1}} \| \nabla  M \|_{H^{s}} \leq \| H_{ext} \|_{H^{s+1}} \mathcal{E}^{\frac{1}{2}}_s (t)\,,
   \end{aligned}
 \end{equation}
 and
 \begin{equation}\label{C3-m}
   \begin{aligned}
     C_{3}^m & \leq C \sum_{\substack{ |m^{\prime}| = 1 }} \| \partial^{m^{\prime }} M \|_{L^{\infty}} \| \Delta \partial^{m-m^{\prime}} M \|_{L^{2}} \| \Delta \partial^{m} M \|_{L^{2}} \\
     & +C \sum_{\substack{ m^{\prime} \leq m, |m^{\prime}|\geq 2}} \| \partial^{m^{\prime }} M \|_{L^{4}} \| \Delta\partial^{m-m^{\prime}} M \|_{L^{4}} \| \Delta \partial^{m} M \|_{L^{2}} \\
     & \leq C \|\nabla M \|^{2}_{H^{s}} \| \Delta M \|_{H^{s}}  \leq C \mathcal{E}_s (t) \mathcal{D}^{\frac{1}{2}}_s (t) \,,
   \end{aligned}
 \end{equation}
 and
 \begin{equation}\label{C4-m}
   \begin{aligned}
     C_{4}^m & = \langle \nabla M \times \partial^{m} H_{ext} + M \times \nabla \partial^{m} H_{ext}, \nabla \partial^{m} M \rangle \\
     & -\sum_{\substack{ 0 \neq m^{\prime} \leq m}} C^{m^{\prime}}_{m} \langle \partial^{m^{\prime}} M \times \partial^{m-m^{\prime}} H_{ext}, \Delta \partial^{m}  M \rangle\\
     & \leq \|\nabla M \|_{L^{\infty}} \| \partial^{m} H_{ext} \|_{L^{2}} \| \nabla \partial^{m} M \|_{L^{2}} + \| \nabla \partial^{m} H_{ext} \|_{L^{2}} \| \nabla \partial^{m} M \|_{L^{2}}\\
     & + C \sum_{\substack{ 0 \neq m^{\prime} \leq m}} \| \partial^{m^{\prime}} M \|_{L^{4}} \| \partial^{m-m^{\prime}} H_{ext} \|_{L^{4}} \| \Delta \partial^{m} M \|_{L^{2}}  \\
     & \leq C \| H_{ext} \|_{H^{s+1}} ( \| \nabla M \|^{2}_{H^{s}} + \| \nabla M \|_{H^{s}} + \| \nabla M \|_{H^{s}} \| \Delta M \|_{H^{s}}) \\
     & \leq C \| H_{ext} \|_{H^{s+1}} ( \mathcal{E}_s (t) + \mathcal{E}^{\frac{1}{2}}_s (t) + \mathcal{E}^{\frac{1}{2}}_s (t) \mathcal{D}^{\frac{1}{2}}_s (t)) \,.
   \end{aligned}
 \end{equation}

 We now turn to estimate the term $C_5^m$ given in \eqref{MEL_10}. We first divide it into two parts according to the form of Lagrangian multiplier $\Gamma (M)$ defined in \eqref{Lagrangian-Multiplier}. Specifically,
 \begin{equation}\label{C5-C51-C52-m}
   \begin{aligned}
     C_{5}^m = \underbrace{ \langle \partial^{m} ( ( H_{ext} \cdot M ) M ), \Delta \partial^{m} M \rangle }_{C_{51}^m} \ \ \underbrace{ - \langle \partial^{m} ( \mid \nabla M \mid^{2} M ), \Delta \partial^{m} M \rangle }_{C_{52}^m} \,.
   \end{aligned}
 \end{equation}
 For the quantity $C_{51}^m$, we estimate that
 \begin{equation}\label{C51-m}
   \begin{aligned}
     C_{51}^m = & -\langle ( \nabla \partial^{m} H_{ext} \cdot M ) M, \Delta \partial^{m} M \rangle - \langle \partial^{m} H_{ext} \cdot \nabla ( M \otimes M), \Delta \partial^{m} M \rangle \\
     & + \sum_{\substack{ 0 \neq m^{\prime} \leq m}} C^{m^{\prime}}_{m} \langle \partial^{m-m^{\prime}} H_{ext} \cdot \partial^{m^{\prime}} ( M \otimes M), \Delta \partial^{m} M \rangle \\
     \leq & \| \nabla \partial^{m} H_{ext} \|_{L^{2}} \| \nabla \partial^{m} M \|_{L^{2}} + \| \partial^{m} H_{ext} \|_{L^{2}} \| \nabla M \|_{L^{\infty}} \| \nabla \partial^{m} M \|_{L^{2}} \\
     & + C \sum_{\substack{ 0 \neq m^{\prime} \leq m}} \| \partial^{m-m^{\prime}} H_{ext} \|_{L^{4}} \| \partial^{m^{\prime}} ( M \otimes M ) \|_{L^{4}} \| \Delta \partial^{m} M \|_{L^{2}} \\
     \leq & C \| H_{ext} \|_{H^{s+1}} \| \nabla M \|_{H^{s}} + C \| H_{ext} \|_{H^{s}} \| \nabla M \|^{2}_{H^{s}} \\
     & + C \| H_{ext} \|_{H^s} ( \| \nabla M \|_{H^s} + \| \nabla M \|^2_{H^s} ) \| \Delta M \|_{H^s} \\
     \leq & C \| H_{ext} \|_{H^{s+1}} ( \mathcal{E}_s (t) + \mathcal{E}^{\frac{1}{2}}_s (t) ) + C \| H_{ext} \|_{H^s} ( \mathcal{E}_s (t) +\mathcal{E}^{\frac{1}{2}}_s (t) ) \mathcal{D}^{\frac{1}{2}}_s (t) \\
     \leq & C \| H_{ext} \|_{H^{s+1}} \big( \mathcal{E}_s (t) + \mathcal{E}^{\frac{1}{2}}_s (t) \big) \big( 1 + \mathcal{D}^{\frac{1}{2}}_s (t) \big) \,,
   \end{aligned}
 \end{equation}
and similarly in \eqref{C51-m} we obtain the bound of $C_{52}^m$ that
\begin{equation}\label{C52-m}
  \begin{aligned}
    C_{52}^m = & - \langle \partial^{m} ( \mid \nabla M \mid^{2} ) M, \Delta \partial^{m} M \rangle - \sum_{\substack{ 0 \neq m^{\prime} \leq m}} C^{m^{\prime}}_{m} \langle \partial^{m-m^{\prime}} ( \mid \nabla M \mid^{2} ) \partial^{m^{\prime}} M, \Delta \partial^{m} M \rangle \\
    \leq & C \| \nabla M \|^{2}_{H^s} \| \Delta M \|_{H^s} + C \| \nabla M \|^3_{H^s} \| \Delta M \|_{H^s} \\
    \leq & C \mathcal{E}_s (t) \mathcal{D}^{\frac{1}{2}}_s (t) + C \mathcal{E}^{\frac{3}{2}}_s (t) \mathcal{D}^{\frac{1}{2}}_s (t) \,.
 \end{aligned}
\end{equation}
As a result, plugging the bounds \eqref{C51-m} and \eqref{C52-m} into \eqref{C5-C51-C52-m} reduces to
\begin{equation}\label{C5-m}
  \begin{aligned}
    C_{5}^m \leq & C \| H_{ext} \|_{H^{s+1}} \big( \mathcal{E}_s (t) + \mathcal{E}^{\frac{1}{2}}_s (t) \big) \big( 1 + \mathcal{D}^{\frac{1}{2}}_s (t) \big) + \big[ \mathcal{E}_s (t) + \mathcal{E}^{\frac{3}{2}}_s (t) \big]  \mathcal{D}^{\frac{1}{2}}_s (t) \,.
 \end{aligned}
\end{equation}
Finally, we substitute the inequalities \eqref{A0-m}, \eqref{A1-m}, \eqref{A2-m}, \eqref{A3-m}, \eqref{A4-m}, \eqref{B1-m}, \eqref{B2-m}, \eqref{C1-m}, \eqref{C2-m}, \eqref{C3-m}, \eqref{C4-m} and \eqref{C5-m} into the equality \eqref{MEL_11} and then we have
\begin{equation*}
  \begin{aligned}
    \tfrac{1}{2} \tfrac{\d}{\d t} \mathcal{E}_s (t) & + \mathcal{D}_s (t)
    \leq C \| H_{ext} \|_{H^{s+1}} \big( \mathcal{E}^\frac{1}{2}_s (t) + \mathcal{E}^{\frac{3}{2}}_s (t) \big) \big( 1 + \mathcal{D}^{\frac{1}{2}}_s (t) \big) + \big[ \mathcal{E}_s (t) + \mathcal{E}^{\frac{3}{2}}_s (t) \big]  \mathcal{D}^{\frac{1}{2}}_s (t) \\
    \leq & C \| H_{ext} \|_{H^{s+1}} \big( \mathcal{E}_s^\frac{1}{2} (t) + \mathcal{E}_s^\frac{3}{2} (t) \big) + C ( 1 + \| H_{ext} \|_{H^{s+1}} ) \big( \mathcal{E}_s^\frac{1}{2} (t) + \mathcal{E}_s^\frac{3}{2} (t) \big) \mathcal{D}_s^\frac{1}{2} (t) \,,
 \end{aligned}
\end{equation*}
and then we finish the proof of Proposition \ref{Prop-AE1}.
\end{proof}

\section{Well-posedness for LLG evolution with a given velocity field}\label{Sec: Local-LLG-Given-v}

In this section, we will justifying the well-posedness of the following wave map type system with a given velocity field $v(t,x) \in \R^d$
\begin{equation}\label{LLC-given-v}
\left\{
\begin{array}{l}
\partial_{t} M + v \cdot \nabla M  =\Delta M+H_{ext}+\Gamma(M)M - M \times ( \Delta M+ H_{ext}) \,,\\[2mm]
M |_{t=0} = M_0 \in \mathbb{S}^{d-1}
\end{array}
\right.
\end{equation}
with the geometric constraint $|M| = 1$, where the Lagrangian multiplier $\Gamma (M)$ is given in \eqref{Lagrangian-Multiplier}. The results of the well-posedness of the LLC equation \eqref{LLC-given-v} will be used in constructing the iterative approximate system of the system \eqref{MEL}. We first illustrate the relations between the Lagrangian multiplier $\Gamma (M)$ and geometric constraint $|M| = 1$. Specifically, if the Lagrangian multiplier $\Gamma (M)$ is of the form \eqref{Lagrangian-Multiplier}, then this constraint will only need to be imposed initially. We then prove the local-in-time existence of \eqref{LLC-given-v} by employing the mollifier method.

\subsection{Lagrangian multiplier $\Gamma (M)$ and constraint $\mid M \mid=1$}

In this subsection, we prove the following lemma on the relation between the Lagrangian multiplier $\Gamma (M)$ and the geometric constraint $\mid M \mid=1$.

\begin{lemma}\label{Lmm-|M|=1}
	Let $s \geq 2$ be a fixed integer and $T \in (0, \infty)$. Given $v \in L^\infty(0,T;H^s(\R^d)) \cap L^2(0,T;H^{s+1}(\R^d))$ and $H_{ext} \in L^\infty ( 0, T; H^{s+1} (\R^d) )$, we assume $M (t,x) \in L^\infty ([0,T] \times \R^d)$, satisfying $ \nabla M \in L^\infty(0,T;H^s(\R^d)$, obeys the LLG equation \eqref{LLC-given-v}.
	
	If the constant $\mid M \mid=1$ is further assumed, then the Lagrange  multiplier $\Gamma (M)$ is of the form \eqref{Lagrangian-Multiplier}. Conversely, if we give the form of $\Gamma (M)$ as in \eqref{Lagrangian-Multiplier} and $M$ satisfies the initial data condition $\mid M_0 \mid=1$, then $\mid M \mid=1$ holds at any time.
\end{lemma}
	
	\begin{proof}[Proof of Lemma \ref{Lmm-|M|=1}]
		From multiplying by $M$ in the LLG equation of \eqref{LLC-given-v}, we easily deduce that
		\begin{equation}\label{Geometric-relation}
		  \begin{aligned}
		    \Gamma(M) \mid M \mid^{2} & = \partial _{t} ( \tfrac{1}{2} \mid M \mid^{2} ) + v \cdot \nabla ( \tfrac{1}{2} \mid M \mid^{2} ) - \Delta M \cdot M - H_{ext} \cdot M \\
		    & = \tfrac{1}{2} ( \partial _{t} + v \cdot \nabla ) \mid M \mid^{2} - \tfrac{1}{2} \Delta( \mid M \mid^{2}) + \mid \nabla M \mid^{2}  -H_{ext} \cdot M \,.
		  \end{aligned}
		\end{equation}
		If $|M| = 1$, the equality \eqref{Geometric-relation} implies that
		$$ \Gamma (M) = |\nabla M|^2 - H_{ext} \cdot M \,. $$
		
		Conversely, if $\Gamma (M)$ is given with the form \eqref{Lagrangian-Multiplier}, the equality \eqref{Geometric-relation} can be rewritten as
		\begin{equation}
		  \begin{aligned}
		    ( \partial _{t} + v \cdot \nabla ) ( \mid M \mid^{2}-1 ) - \Delta ( \mid M \mid^{2}-1 ) = 2 \Gamma(M) ( \mid M \mid^{2}-1) \,.
		  \end{aligned}
		\end{equation}
		Let $h=\mid M \mid^{2}-1$. Then, by the initial assumption $|M_0| = 1$, one knows that $h$ solves the following Cauchy problem
		\begin{equation}\label{Evolution-h}
		  \left\{
		    \begin{array}{l}
		      \partial_{t} h + v \cdot \nabla h - \Delta h = 2 \Gamma(M) h \,,\\[2mm]
		      h(0,x) = \mid M_0 \mid^{2}-1=0 \,.
		    \end{array}
		  \right.
		\end{equation}
		From multiplying \eqref{Evolution-h} by $h$ and integrating by parts over $x \in \R^d$, we deduce that
		\begin{equation*}
		  \begin{aligned}
		    \tfrac{1}{2} \tfrac{\d}{\d t} \| h \|^{2}_{L^{2}} + \| \nabla h \|^{2}_{L^{2}} = 2 \langle \Gamma(M) h, h \rangle - \tfrac{1}{2} \langle \nabla \cdot v, h^2 \rangle \\
		    \leq \big( 2 \| \Gamma(M) \|_{L^{\infty}} + \tfrac{1}{2} \| \nabla \cdot v \|_{L^\infty} \big) \| h \|^{2}_{L^{2}} \,.
		  \end{aligned}
		\end{equation*}
		Noticing that $ v \in L^\infty (0,T;H^s) \cap L^2 (0,T;H^{s+1}) $, $H_{ext} \in L^\infty (0,T;H^{s+1})$, $M \in L^\infty ( [0,T] \times \R^d )$ and $\nabla M \in L^\infty (0,T;H^s)$, we derive from the Sobolev embedding $H^2 (\R^d) \hookrightarrow L^\infty (\R^d)$ that
		\begin{equation*}
		  \begin{aligned}
		    2 & \| \Gamma(M) \|_{L^{\infty}} + \tfrac{1}{2} \| \nabla \cdot v \|_{L^\infty} \leq C \| \ |\nabla M|^2 - H_{ext} \cdot M \|_{L^\infty} + C \| \nabla v \|_{L^\infty} \\
		    \leq & C \| \nabla M \|^2_{H^2} + C \| M \|_{L^\infty ([0,T] \times \R^d)} \| H_{ext} \|_{H^2} + C \| \nabla v \|_{H^2} \\
		    \leq & C \big( \| \nabla M \|^2_{H^2} + \| M \|_{L^\infty ([0,T] \times \R^d)} \| H_{ext} \|_{H^2} + \| \nabla v \|_{H^2} \big) (t) \equiv \tfrac{1}{2} \Lambda (t) \in L^1 ([0,T]) \,.
		  \end{aligned}
		\end{equation*}
		We thereby have
		\begin{equation}
		  \begin{aligned}
		    \tfrac{\d}{\d t} \| h \|^2_{L^2} \leq \Lambda (t) \| h \|^2_{L^2} \,,
		  \end{aligned}
		\end{equation}
		which implies by the Gr\"onwall inequality and the vanishing of initial data $h (0,x)$ that
		\begin{equation}
		  \begin{aligned}
		    0 \leq \| h (t) \|^2_{L^2} \leq \| h (0) \|^2_{L^2} \exp \Big( \int_0^t \Lambda (s) \d x \Big) = 0
		  \end{aligned}
		\end{equation}
		for all $t \in [0,T]$. Thus, we have $h(x,t)=0$ holds for all times $t \in [0,T]$. Consequently, the proof of Lemma \ref{Lmm-|M|=1} is finished.
	\end{proof}

\subsection{Local existence of LLG equation with a given background velocity}

In this subsection, we will prove the existence of the local-in-time classical solutions to the LLG equation \eqref{LLC-given-v} with a given bulk velocity $v$ by adopting the mollifier approximate method.

\begin{proposition}\label{Prop-Local-LLG}
  For the integer $s \geq 2$, and $T_0 > 0$, let vector fields $v \,, H_{ext} \in \R^d$  satisfy $ v \in L^\infty(0,T_0;H^s(\R^d)) \cap L^2(0,T_0 ; H^{s+1}(\R^d))$, $ H_{ext} \in L^\infty(0,T_0 ; H^s (\R^d) $. If the initial data $M_0 \in \mathbb{S}^{d-1}$ subjects to $\nabla M_0 \in H^s (\R^d)$, then there is a $T \in (0,T_0]$, depending only on $M_0$, $H_{ext}$ and $v$, such that the Cauchy problem \eqref{LLC-given-v} has a unique classical solution $M(t,x) \in \mathbb{S}^{d-1}$ satisfying $\nabla M \in C(0,T;H^s(\R^d), \Delta M \in L^2(0,T;H^s(\R^d)$. Moreover, there exists a positive $C_0$, depending only on $M_0 $, $H_{ext}$, $v$ and $T_0$, such that following bound holds:
  \begin{equation}\label{bound}
    \begin{aligned}
      \| \nabla M \|^2_{L^\infty (0,T; H^s) } + \| \Delta M \|^2_{L^2 (0,T; H^s)} \leq C_0 \,.
    \end{aligned}
  \end{equation}
\end{proposition}

\begin{proof}[Proof of Proposition \ref{Prop-Local-LLG}]
	We divide three steps to complete the proof: 1) Construct the approximate system by the mollifier approximate method; 2) Derive a uniform energy bound of the approximate system on a uniform time interval $[0,T]$; 3) By compactness arguments, we take the limit from the approximate system to obtain the solutions to the Cauchy problem \eqref{LLC-given-v}.
	
{\em Step 1:Construct the approximate system.} We first define a mollifier
$$ \J_\eps f := \mathcal{F}^{-1} \big( \mathbf{1}_{|\xi| \leq \frac{1}{\eps}} \mathcal{F} (f) \big) \,,  $$
where the symbol $\mathcal{F}$ stands for the Fourier transform operator and $\mathcal{F}^{-1}$ is its inverse transform. We remark that the mollifier $\J_\eps$ has property $\J^2_\eps = \J_\eps$. Then the approximate system of \eqref{MEL} can be constructed as follows:
\begin{equation}\label{ch5-2}
  \left\{
    \begin{array}{l}
      \partial_{t} M^{\varepsilon} = - \J_\eps ( v \cdot \nabla \J_\eps M^{\varepsilon} ) + \Delta \J_\eps M^{\varepsilon} + H_{ext} \\
      \qquad \qquad + \J_\eps [ \Gamma( \J_\eps M^{\varepsilon}) M^{\varepsilon} ] - \J_\eps [ \J_\eps M^{\varepsilon} \times ( \Delta \J_\eps M^{\varepsilon} + H_{ext} )] \,, \\[1.5mm]
      M^{\varepsilon}|_{t=0} = \J_\eps M_0 \,.
    \end{array}
  \right.
\end{equation}
From the ODE theory, we deduce that for any fixed $\eps > 0$, there exists a maximal $ T_\varepsilon \in (0,T_0]$ such that the approximate system \eqref{ch5-2} admits a unique solution $M^\eps (t,x)$ satisfying $ M^{\varepsilon} \in L^\infty ([0,T_\eps ) \times \R^d) $ and $ \nabla M^\eps \in C(0,T_{\varepsilon};H^s)$. Since $\J^2_\eps = \J_\eps$, we observe that $M^{\varepsilon}$ also solves
\begin{equation}\label{ch5-3}
  \left\{
    \begin{array}{l}
      \partial_{t} M^{\varepsilon}= -\J_\eps( v \cdot \nabla  M^{\varepsilon} )+ \Delta  M^{\varepsilon}+H_{ext}+\J_\eps [\gamma( M^{\varepsilon})M^{\varepsilon} ]- \J_\eps[ M^{\varepsilon} \times ( \Delta M^{\varepsilon}+ H_{ext})]\,,\\[1.5mm]
      M^{\varepsilon}|_{t=0} = \J_\eps M_0 \,.
    \end{array}
  \right.
\end{equation}

{\em Step 2: Uniform energy estimate.} First, we calculate the $L^2$-estimate of the approximate system \eqref{ch5-3}. Multiplying by $\Delta M^{\varepsilon}$ in the first equation of \eqref{ch5-3} and integrating by parts over $x \in \R^d$, we have
 \begin{equation}\label{L2-M}
   \begin{aligned}
     \tfrac{1}{2} \tfrac{\d}{\d t} \| \nabla   M^{\varepsilon} \|^{2}_{L^{2}} + \| \Delta  M^{\varepsilon} \|^{2}_{L^{2}} = & \underset{I_1}{\underbrace{ \langle ( v \cdot \nabla M^{\varepsilon} ), \Delta  M^{\varepsilon} \rangle }} + \underset{I_2}{\underbrace{ \langle ( \Gamma(M^{\varepsilon}) M^{\varepsilon} ), \Delta  M^{\varepsilon} \rangle }} \\
     & + \underset{I_3}{\underbrace{ \langle ( M^{\varepsilon} \times H_{ext} ), \Delta   M^{\varepsilon} \rangle }} - \underset{I_4}{\underbrace{ \langle H_{ext}, \Delta  M^{\varepsilon} \rangle }} \,.
   \end{aligned}
 \end{equation}
By H\"older inequality and Sobolev embedding theory, we estimate that
 \begin{equation*}
   \begin{aligned}
     I_1 \leq & \| v \|_{L^{\infty}} \| \nabla M^{\varepsilon} \|_{L^{2}} \| \Delta M^{\varepsilon} \|_{L^{2}} \leq \| v \|_{H^{2}} \| \nabla M^{\varepsilon} \|_{L^{2}} \| \Delta M^{\varepsilon} \|_{L^{2}} \,, \\
     I_2 \leq & C \| M^{\varepsilon} \|_{L^{\infty}} \| \nabla M^{\varepsilon} \|^{2}_{H^{1}} \| \Delta M^{\varepsilon} \|_{L^{2}} + \| M^{\varepsilon} \|^{2}_{L^{\infty}} \| H_{ext} \|_{L^{2}} \| \Delta M^{\varepsilon} \|_{L^{2}} \,, \\
     I_3 \leq & C \| M^{\varepsilon} \|_{L^{\infty}} \| H_{ext} \|_{L^{2}} \| \Delta M^{\varepsilon} \|_{L^{2}} \,, \\
     I_4 \leq & C  \| H_{ext} \|_{L^{2}} \| \Delta M^{\varepsilon} \|_{L^{2}} \,.
   \end{aligned}
 \end{equation*}
Then, from plugging the previous four inequalities of $I_i$ $(1 \leq i \leq 4)$ into the equality \eqref{L2-M}, we deduce the $L^2$-estimate
 \begin{equation}\label{L2}
  \begin{aligned}
    \tfrac{1}{2} \tfrac{\d}{\d t} \| \nabla   M^{\varepsilon} \|^{2}_{L^{2}} + \| \Delta  M^{\varepsilon} \|^{2}_{L^{2}} \leq & C \Big( \| v \|_{H^{2}} \| \nabla M^{\varepsilon} \|_{L^{2}} + \| M^{\varepsilon} \|_{L^{\infty}} \| \nabla M^{\varepsilon} \|^{2}_{H^{1}} \\
    & + \| M^{\varepsilon} \|^{2}_{L^{\infty}} \| H_{ext} \|_{L^{2}} + \| H_{ext} \|_{L^{2}} \Big) \| \Delta M^{\varepsilon} \|_{L^{2}} \,.
  \end{aligned}
\end{equation}

Second, we estimate the higher order energy bounds of the approximate system \eqref{ch5-3}. For all multi-indexes $m \in \mathbb{N}^d$ with $1\leq |m| \leq N$, we act the $m$-order derivative operator $\partial^{m}$ on the first equation of the system \eqref{ch5-3}, take $L^{2}$-inner product via multiplying by $\Delta \partial^{m} M^{\varepsilon}$ and integrate by parts over $x \in \R^d$. Then we have
 \begin{equation}\label{higher}
   \begin{aligned}
     & \tfrac{1}{2} \tfrac{\d}{\d t} \| \nabla  \partial^{m} M^{\varepsilon} \|^{2}_{L^{2}} + \| \Delta \partial^{m} M^{\varepsilon} \|^{2}_{L^{2}} = \underset{I\!I_1}{\underbrace{ \langle \partial^{m} ( v \cdot \nabla M^{\varepsilon} ) , \Delta \partial^{m} M^{\varepsilon} \rangle }} \\
     & + \underset{I\!I_2}{\underbrace{ \langle \partial^{m} ( \Gamma ( M^{\varepsilon} ) M^{\varepsilon}) , \Delta \partial^{m} M^{\varepsilon} \rangle}}  + \underset{I\!I_3}{\underbrace{ \langle \partial^{m} ( M^{\varepsilon} \times \Delta M^{\varepsilon} ), \Delta \partial^{m}  M^{\varepsilon} \rangle }} \\
     & + \underset{I\!I_4}{\underbrace{ \langle \partial^{m} ( M^{\varepsilon} \times H_{ext} ), \Delta \partial^{m}  M^{\varepsilon} \rangle }} - \underset{I\!I_5}{\underbrace{ \langle \partial^{m} H_{ext} , \Delta \partial^{m} M^{\varepsilon} \rangle}} \,.
   \end{aligned}
 \end{equation}
Now we estimate the quantities in the right hand side of the equality \eqref{higher} term by term. From the H\"older inequality and Sobolev embedding theory, one easily derives that
\begin{equation*}
  \begin{aligned}
    I\!I_1 = & \langle \partial^{m} v \cdot \nabla M^{\varepsilon}, \Delta \partial^{m} M^{\varepsilon} \rangle + \langle v \cdot \nabla \partial^{m} M^{\varepsilon} , \Delta \partial^{m} M^{\varepsilon} \rangle \\
    & + \sum_{\substack{ 0 \neq m^{\prime} \leq m}} C^{m^{\prime}}_{m} \langle \partial^{m-m^{\prime}} v \nabla \partial^{m^{\prime}} M^{\varepsilon}, \Delta \partial^{m} M^{\varepsilon} \rangle \\
    & \leq \| v \|_{H^s} \| \nabla M^{\varepsilon} \|_{H^s} \| \Delta M^{\varepsilon} \|_{H^s} \,.
  \end{aligned}
\end{equation*}
Similarly as the above estimate, by utilizing the H\"older inequality and the Sobolev embedding theory, we have
\begin{equation*}
  \begin{aligned}
    I\!I_2 \leq & C ( \| M^{\varepsilon} \|^2_{L^{\infty}} + \| \nabla M^{\varepsilon} \|^2_{H^s} ) ( \| H_{ext} \|_{H^s} + \| \nabla M^{\varepsilon} \|_{H^s} ) \| \Delta \partial^{m} M^{\varepsilon} \|_{L^{2}} \,,
  \end{aligned}
\end{equation*}
and
\begin{equation*}
  \begin{aligned}
    I\!I_3 \leq & C \| \nabla M^{\varepsilon} \|^2_{H^s} \| \Delta \partial^{m} M^{\varepsilon} \|_{L^{2}} \,,\\
    I\!I_4 \leq & C ( \| M^{\varepsilon} \|_{L^{\infty}} + \| \nabla M^{\varepsilon} \|_{H^s} ) \| H_{ext} \|_{H^s} \| \Delta \partial^{m} M^{\varepsilon} \|_{L^{2}} \,, \\
    I\!I_5 \leq & C \| H_{ext} \|_{H^s} \| \Delta \partial^{m} M^{\varepsilon} \|_{L^{2}} \,.
  \end{aligned}
\end{equation*}
Therefore, by combining all of the above inequalities of $I\!I_i$ $(1 \leq i \leq 5)$, we obtain
\begin{equation}\label{higher-1}
  \begin{aligned}
    \tfrac{1}{2} \tfrac{\d}{\d t} & \| \nabla  \partial^{m} M^{\varepsilon} \|^{2}_{L^{2}} + \| \Delta \partial^{m} M^{\varepsilon} \|^{2}_{L^{2}} \\
    \leq & C ( 1 + \| M^{\varepsilon} \|^2_{L^{\infty}} + \| \nabla M^{\varepsilon} \|^2_{H^s} ) \| H_{ext} \|_{H^s} \| \Delta \partial^{m} M^{\varepsilon} \|_{L^{2}} \\
    + & C ( \| v \|_{H^s} + \| \nabla M^{\varepsilon} \|^2_{H^s} + \| M^{\varepsilon} \|^2_{L^{\infty}} + \| \nabla M^{\varepsilon} \|_{H^s} ) \| \nabla M^{\varepsilon} \|_{H^s} \| \Delta \partial^{m} M^{\varepsilon} \|_{L^{2}}
  \end{aligned}
\end{equation}
holds for all $1 \leq |m| \leq N$. Combining the $L^2$-estimate \eqref{L2} and the higher order derivative estimate \eqref{higher-1}, we have
\begin{equation*}
  \begin{aligned}
    &\tfrac{1}{2} \tfrac{\d}{\d t} \| \nabla  M^{\varepsilon} \|^{2}_{H^s} + \| \Delta  M^{\varepsilon} \|^{2}_{H^s}
    \leq C ( 1 + \| M^{\varepsilon} \|^2_{L^{\infty}} + \| \nabla M^{\varepsilon} \|^2_{H^s} ) \| H_{ext} \|_{H^s} \| \Delta M^{\varepsilon} \|_{H^s} \\
    & + C ( \| v \|_{H^s} + \| \nabla M^{\varepsilon} \|^2_{H^s} + \| M^{\varepsilon} \|^2_{L^{\infty}} + \| \nabla M^{\varepsilon} \|_{H^s} ) \| \nabla M^{\varepsilon} \|_{H^s} \| \Delta  M^{\varepsilon} \|_{H^s} \,.
  \end{aligned}
\end{equation*}
We further use the Young's inequality and then derive that
\begin{equation}\label{higher-order}
  \begin{aligned}
    \tfrac{1}{2} \tfrac{\d}{\d t} \| \nabla  M^{\varepsilon} & \|^{2}_{H^s} + \| \Delta  M^{\varepsilon} \|^{2}_{H^s}
    \leq C ( 1 + \| M^{\varepsilon} \|^4_{L^{\infty}} + \| \nabla M^{\varepsilon} \|^4_{H^s} ) \| H_{ext} \|^{2}_{H^s} \\
    + & C ( \| v \|^{2}_{H^s} + \| \nabla M^{\varepsilon} \|^4_{H^s} + \| M^{\varepsilon} \|^4_{L^{\infty}} + \| \nabla M^{\varepsilon} \|^{2}_{H^s} ) \| \nabla M^{\varepsilon} \|^{2}_{H^s} \,.
  \end{aligned}
\end{equation}

Third, we notice that the norm $\| M^{\varepsilon} \|_{L^{\infty}}$ in the above $H^s$-estimate \eqref{higher-order} is not controlled yet. To deal with it, we estimate as follows:
\begin{equation}\label{M-L-infty}
  \begin{aligned}
    \| M^{\varepsilon} \|_{L^{\infty}} \leq & \| M^{\varepsilon} - \J_\eps M_0 \|_{L^{\infty}} + \| \J_\eps M_0 \|_{L^{\infty}} \\
    \leq & C \| M^{\varepsilon} -\J_\eps M_0 \|_{H^{2}} + 1 \\
    = & C \| M^{\varepsilon} -\J_\eps M_0 \|_{L^{2}} + C \| \nabla M^{\varepsilon} - \J_\eps M_0 \|_{H^{1}} + 1 \\
    = & C \| M^{\varepsilon} - \J_\eps M_0 \|_{L^{2}} + C \| \nabla M^{\varepsilon} \|_{H^{1}} + C \| \J_\eps M_0 \|_{H^{1}} +  1\\
    \leq & \| M^{\varepsilon} - \J_\eps M_0 \|_{L^{2}} + C \| \nabla M^{\varepsilon} \|_{H^s} + C \| M_0 \|_{H^s} + 1 \,,
  \end{aligned}
\end{equation}
where the relation $\|\J_\eps M_0 \|_{L^{\infty}} = | M_0 | =1 $ and the Sobolev embedding $H^2 (\R^d) \hookrightarrow L^\infty (\R^d)$ are utilize. It remains to control the norm $\| M^{\varepsilon} - \J_\eps M_0 \|_{L^{2}}$, which vanishes at $t = 0$. To be more precise, from the LLG equation \eqref{LLC-given-v}, we deduce that
\begin{equation}\label{M-1}
  \begin{aligned}
    \tfrac{1}{2} & \tfrac{\d}{\d t} \| M^{\varepsilon} - \J_\eps M_0 \|_{L^{2}} = \langle \partial_{t} M^{\varepsilon}, M^{\varepsilon} - \J_\eps M_0 \rangle \\
    & = \langle - \J_\eps ( v \cdot \nabla  M^{\varepsilon} ) + \Delta  M^{\varepsilon} + H_{ext} + \J_\eps [ \Gamma ( M^{\varepsilon} ) M^{\varepsilon} ] \\
    & - \J_\eps [ M^{\varepsilon} \times ( \Delta M^{\varepsilon} + H_{ext} ) ] , M^{\varepsilon} - \J_\eps M_0 \rangle \\
    &\leq C ( 1 + \| v \|_{H^s} + \| \nabla M^{\varepsilon} \|_{H^s} + \| M^{\varepsilon} - \J_\eps M_0 \|_{L^{2}} ) \| \nabla M^{\varepsilon} \|_{H^s} \| M^{\varepsilon} - \J_\eps M_0 \|_{L^{2}} \\
    & + C \| M^{\varepsilon} - \J_\eps M_0 \|^2_{L^{2}} (  1 + \| M^{\varepsilon} - \J_\eps M_0 \|_{L^{2}} ) + C ( 1 + \| M^{\varepsilon} - \J_\eps M_0 \|^2_{L^{2}} \\
    & + \| \nabla M^{\varepsilon} \|^2_{H^s} ) \|  H_{ext} \|_{H^s} \| M^{\varepsilon} - \J_\eps M_0 \|_{L^{2}} \,.
  \end{aligned}
\end{equation}
By the inequality \eqref{M-L-infty}, the  higher order estimate \eqref{higher-order} can be written
\begin{equation}\label{higher-order-new}
  \begin{aligned}
    & \tfrac{1}{2} \tfrac{\d}{\d t} \| \nabla  M^{\varepsilon} \|^{2}_{H^s} + \| \Delta  M^{\varepsilon} \|^{2}_{H^s} \leq C ( 1 + \| M^{\varepsilon} - \J_\eps M_0 \|^4_{L^{2}} + \| \nabla M^{\varepsilon} \|^4_{H^s} ) \| H_{ext} \|^{2}_{H^s} \\
    & + C ( \| v \|^{2}_{H^s} + \| \nabla M^{\varepsilon} \|^4_{H^s} + \| M^{\varepsilon} - \J_\eps M_0 \|^4_{L^{2}} + \| \nabla M^{\varepsilon} \|^{2}_{H^s} ) \| \nabla M^{\varepsilon} \|^{2}_{H^s} \,.
  \end{aligned}
\end{equation}
Combining \eqref{M-1} and \eqref{higher-order-new}, we have
\begin{equation}\label{5.12}
  \begin{aligned}
    & \tfrac{\d}{\d t} ( \| \nabla  M^{\varepsilon} \|^{2}_{H^s} + \| M^{\varepsilon} - \J_\eps M_0 \|_{L^{2}} ) + \| \Delta  M^{\varepsilon} \|^{2}_{H^s} \\
    & \leq C ( 1 + \| M^{\varepsilon} - \J_\eps M_0 \|^4_{L^{2}} + \| \nabla M^{\varepsilon} \|^4_{H^s} ) \| H_{ext} \|^{2}_{H^s} \\
    & + C ( \| v \|^{2}_{H^s} + \| \nabla M^{\varepsilon} \|^4_{H^s} + \| M^{\varepsilon} - \J_\eps M_0 \|^4_{L^{2}} + \| \nabla M^{\varepsilon} \|^{2}_{H^s} ) \| \nabla M^{\varepsilon} \|^{2}_{H^s} \\
    & + C ( 1 + \|  v \|_{H^s} + \|  \nabla M^{\varepsilon} \|_{H^s} + \| M^{\varepsilon} - \J_\eps M_0 \|_{L^{2}} ) \|  \nabla M^{\varepsilon} \|_{H^s} \| M^{\varepsilon} - \J_\eps M_0 \|_{L^{2}} \\
    & + C \| M^{\varepsilon} - \J_\eps M_0 \|^2_{L^{2}} ( 1 + \| M^{\varepsilon} - \J_\eps M_0 \|_{L^{2}} ) + C ( 1 + \| M^{\varepsilon} - \J_\eps M_0 \|^2_{L^{2}} \\
    & + \| \nabla M^{\varepsilon} \|^2_{H^s} ) \|  H_{ext} \|_{H^s} \| M^{\varepsilon} - \J_\eps M_0 \|_{L^{2}} \,.
  \end{aligned}
\end{equation}
We introduce the approximate energy functional
\begin{equation*}
  \begin{aligned}
    E_{\varepsilon} (t) = \| \nabla M^{\varepsilon} \|^2_{H^s} + \| M^{\varepsilon} - \J_\eps M_0 \|^2_{L^{2}}
  \end{aligned}
\end{equation*}
and the approximate energy dissipative rate functional
\begin{equation*}
  \begin{aligned}
    D_{\varepsilon} (t) = \| \Delta M^{\varepsilon} \|^2_{H^s} \,.
  \end{aligned}
\end{equation*}
Then the inequality \eqref{5.12} can be rewritten as
\begin{equation}\label{higher-order-energy}
  \begin{aligned}
    & \tfrac{\d}{\d t} E_{\varepsilon} (t) + D_{\varepsilon} (t) \leq C ( 1 + \| v \|^{2}_{H^s} + \| H_{ext} \|^2_{H^s} ) ( 1 + E_{\varepsilon} (t) )^3
  \end{aligned}
\end{equation}
holds for all $t \in [0, T_\varepsilon)$.

Finally, we derive the uniform bounds of the energy functional $E_\varepsilon(t)$ by using the Gr\"onwall arguments. Notice that
\begin{equation*}
\begin{aligned}
E_\varepsilon (0) = \| \nabla M^{\varepsilon} \|^2_{H^s} = \| \nabla \J_\eps M_0 \|^2_{H^s} \leq \| \nabla M_0 \|^2_{H^s} : = E_0 \,.
\end{aligned}
\end{equation*}
We now define
\begin{equation*}
  T^1_\varepsilon = \sup \Big\{ \tau \in [0, T_\varepsilon) ; \sup_{t \in [0,\tau]} E_\varepsilon (t) \leq 2 E_0 \Big\} \geq 0 \,.
\end{equation*}
By the continuity of $E_\eps (t)$ on $[0, T_\varepsilon )$, we know that $T^1_\varepsilon > 0$. Then the inequality \eqref{higher-order-energy} implies that for all $t \in [0,  T^1_\varepsilon )$
\begin{equation*}
  \begin{aligned}
    & \tfrac{\d}{\d t} E_{\varepsilon} (t) + D_{\varepsilon} (t) \leq C ( 1 + \| v \|^{2}_{H^s} + \| H_{ext} \|^2_{H^s} ) ( 1 + 2 E_0 )^2 \big( 1 + E_{\varepsilon} (t) \big) \leq \Upsilon (t) \big( 1 + E_{\varepsilon} (t) \big) \,,
  \end{aligned}
\end{equation*}
where $\Upsilon (t) = C ( 1 + 2 E_0 )^2 ( 1 + \| v (t) \|^{2}_{H^s} + \| H_{ext} (t) \|^2_{H^s} )  \in L^\infty (0, T_0)$ by the conditions of Proposition \ref{Prop-Local-LLG}. Then, by the Gr\"onwall inequality , we have
\begin{equation*}
  \begin{aligned}
    E_\varepsilon(t) \leq & (E_0 + \int_ 0^t \Upsilon(\tau) \d \tau ) \exp \Big( \int_0^t \Upsilon (\tau) \d \tau \Big) \\
    \leq & ( E_0 + \| \Upsilon \|_{L^\infty(0, T_0)} t ) \exp ( \| \Upsilon \|_{L^\infty(0, T_0)} t ) : = G(t) \,,
  \end{aligned}
\end{equation*}
where the function $G(t)$, independent of $\varepsilon$, is increasing on $ [0, T^1_\varepsilon]$ and $G (0) = E_0$. Consequently, there exists a $ T > 0$, independent of $\varepsilon$, such that for all $ t \in [0,T]$, $G(t) \leq 2E_0$. By the definition of $T^1_\varepsilon$, we know that $T^1_\varepsilon \geq T >0$. In other words, for all $\varepsilon > 0$ and $t \in [0,T]$, we have $E_\varepsilon(t) \leq 2 E_0$. Then we obtain the following uniform energy bound
\begin{equation}\label{b}
  \begin{aligned}
    E_\varepsilon (t) + \int_ 0^t D_\varepsilon(\tau ) \d \tau \leq ( 1 + 2 E_0 ) \| \Upsilon \|_{L^\infty (0,T_0)} t \leq ( 1 + 2 E_0 ) \| \Upsilon \|_{L^\infty (0,T_0)} T : = c_0 T
  \end{aligned}
\end{equation}
holds for all $ t \in [0,T]$.

{\em Step 3: Pass to the limits.} By the bounds \eqref{M-L-infty} and \eqref{b} , we have
\begin{equation}\label{M-11}
  \begin{aligned}
    \| M^{\varepsilon}\|_{L^{\infty}([0,T] \times \R^d)} \leq & C \| M^{\varepsilon} - \J_\eps M_0 \|_{L^{2}} + C \| \nabla M^{\varepsilon} \|_{H^s} + C \| \nabla M_0 \|_{H^s} + 1 \\
    \leq & C E^{\frac{1}{2}}_{\varepsilon} (t) + C \| \nabla M_0 \|_{H^s} + 1 \leq C \sqrt{c_0 T} + C \| \nabla M_0 \|_{H^s} + 1 \,.
  \end{aligned}
\end{equation}
We thereby derive from the uniform bounds \eqref{b} and \eqref{M-11} that there exists a $M \in L^\infty ([0,T]\times \R^3)$ satisfying  $\nabla M \in C(0,T;H^s (\R^d) , \Delta M \in L^2(0,T;H^s(\R^d)$ such that
\begin{equation*}
  \begin{aligned}
    & M^\eps \rightarrow M \quad \textrm{weakly-}\star \ \ \textrm{in} \ \ L^\infty ( [0,T] \times \R^d ) \,, \\
    & \nabla M^\eps \rightarrow \nabla M \quad \textrm{weakly-}\star \ \ \textrm{in}\ \ t \geq 0, \textrm{weakly  in}\ \ H^s (\R^d) \,, \\
    & \Delta M^\eps \rightarrow \Delta M \quad \textrm{weakly in}\ \ L^2 ( 0, T; H^s (\R^d) )
  \end{aligned}
\end{equation*}
as $\eps \rightarrow 0$, and $M$ obeys the first equation of \eqref{LLC-given-v} after passing limits in the approximate system \eqref{ch5-3} as $\varepsilon \rightarrow 0 $. More precisely,
\begin{equation*}
  \begin{aligned}
    & \partial_{t} M + v \cdot \nabla M  = \Delta M + H_{ext} + \Gamma (M) M - M \times ( \Delta M + H_{ext} ) \\
    & \Gamma (M) = | \nabla M |^2 - H_{ext} \cdot M , \ M \in L^\infty ( [0,T] \times \R^d) \,.
  \end{aligned}
\end{equation*}
Then Lemma \ref{Lmm-|M|=1} tells us that $M \in S^{d-1}$ holds for all $t \in [0,T]$. Furthermore, \eqref{higher-order-energy} and \eqref{b} imply that energy functional $E_\varepsilon(t)$ is uniformly bounded and equicontinuous in [0,T]. By Arzel\`a-Ascoli theorem, we deduce that $\nabla M \in C(0,T;H^s (\R^d))$ and then the proof of Proposition \ref{Prop-Local-LLG} is finished.
\end{proof}

\section{Local well-posedness with large initial data}\label{Sec: Local-Result}

In this section, we prove the local well-posedness of the evolutionary model for magnetoelastic system \eqref{MEL} with large initial data. We first carefully design a nonlinear iterative approximate system, where the nonlinearity is due to the geometric constraint $|M^{n+1}| = 1$. Then, we derive a uniform energy bound of the iterative approximate system on a time interval $[0,T]$, where $T > 0$ is independent of $n \geq 0$. Finally, by standard compactness arguments, we can prove the local well-posedness of the Cauchy problem \eqref{MEL}-\eqref{IC-MEL}.

\subsection{The iterative approximate system}
In this subsection, we construct the approximate system by iteration. More precisely, the iterative approximate system is constructed as follows: for all integer $n\geq 0$
\begin{equation}\label{MEL-app}
  \left\{
    \begin{array}{l}
      \partial_t v^{n+1} + v^{n} \cdot \nabla v^{n} + \nabla p^{n+1} + \nabla \cdot (\nabla M^{n} \odot \nabla M^{n} - F^{n} (F^{n})^{\top}) \\
      \qquad \qquad \qquad = \nu \Delta v^{n+1} + \nabla H_{ext} M^{n}  \,, \\
      \nabla \cdot v^{n+1}= 0 \,, \\
      \partial_t F^{n+1} + v^{n} \cdot \nabla F^{n} = \nabla v^{n} F^{n} \,, \\
      \partial_t M^{n+1} + v^{n} \cdot \nabla M^{n+1}  =\Delta M^{n+1} + H_{ext} + \Gamma(M^{n+1}) M^{n+1} \\
      \qquad \qquad \qquad \qquad \qquad \ \ \ - M^{n+1} \times ( \Delta M^{n+1} + H_{ext})\,,\\
      | M^{n+1} | = 1 \,,\\
      (v^{n+1}, F^{n+1}, M^{n+1} ) |_{t=0} = ( v_0(x), F_0(x), M_0(x) ) \in \R^d \times \R^d \times \mathbb{S}^{d-1} \,.
    \end{array}
  \right.
\end{equation}
The iterative starts form $n=0$ with
\begin{equation*}
  \begin{aligned}
    (v^{0} (t,x), F^{0} (t,x), M^{0} (t,x) ) = (v_0 (x), F_0 (x), M_0 (x) ) \,.
  \end{aligned}
\end{equation*}

We first give the existence result of the iterative approximate system \eqref{MEL-app} as follows:

\begin{lemma}\label{Lmm-Exist-Appr}
	Suppose that $s \geq 2$ and the initial data $(v_0, F_0, M_0) \in \R^d \times \R^{d \times d} \times S^{d - 1}$ satisfies $v_0, F_0, \nabla M_0 \in H^s$. Then there is a maximal number $T^*_{n+1} > 0$ such that the system \eqref{MEL-app} admits a unique solution $(v^{n+1}, F^{n+1}, M^{n+1})$ satisfying $ v^{n+1} \in C(0,T^*_{n+1}; H^s ) \cap L^2(0,T^*_{n+1}; H^{s+1})$ and $ F^{n+1}, \nabla M^{n+1} \in C(0,T^*_{n+1}; H^s )$.
\end{lemma}

\begin{proof}[Proof of Lemma \ref{Lmm-Exist-Appr}]
	For the case $n+1$, the vector-valued functions $v^n, M^n$ and the matrix-valued function $F^n$ are known. That is, the velocity equation of $v^{n+1}$ is a linear stokes type system,
	\begin{equation}\label{MEL-4}
	  \left\{
	    \begin{array}{l}
	      \partial_t v^{n+1} + v^{n} \cdot \nabla v^{n} + \nabla p^{n+1} + \nabla \cdot (\nabla M^{n} \odot \nabla M^{n} - F^{n} (F^{n})^{\top}) \\
	      \qquad \qquad \qquad = \nu \Delta v^{n+1} + \nabla H_{ext} M^{n}  \,, \\
	      \nabla \cdot v^{n+1}= 0 \,, \\
	      v^{n+1}|_{t=0} = v_0 \in \R^d \,,
	    \end{array}
	  \right.
	\end{equation}
	which admits a unique solution $ v^{n+1} \in C(0,\widehat{T}_{n+1}; H^s) \cap L^2(0,\widehat{T}_{n+1}; H^{s+1})$ on the maximal time interval $[0,\widehat{T}_{n+1} )$. Moreover, the matrix-valued function $F^{n+1}$ obeys a linear ODE system
	\begin{equation}
	  \left\{
	    \begin{array}{l}
	      \partial_t F^{n+1} + v^{n} \cdot \nabla F^{n} = \nabla v^{n} F^{n} \,, \\
	      F^{n+1} |_{t=0} = F_0 \in \R^{d \times d} \,,
	    \end{array}
	  \right.
	\end{equation}
	in which the spatial variables can be regarded as the parameters. Thus the evolution of $F^{n+1}$ admits a unique solution on the maximal interval $[ 0, \overline{T}_{n+1} )$. By the regularities of $v^n$, $F^n$ and $F_0$, one easily derives $ F^{n+1} \in C( 0, \overline{T}_{n+1} ; H^s ) $. Finally, the orientation equation of $M^{n+1}$ is the LLG equation with a given bulk velocity $v^{n}$
	\begin{equation}\label{MEL-5}
	  \left\{
	    \begin{array}{l}
	      \partial_t M^{n+1} + v^{n} \cdot \nabla M^{n+1}  =\Delta M^{n+1} + H_{ext} + \Gamma(M^{n+1}) M^{n+1} \\
	      \qquad \qquad \qquad \qquad \qquad \ \ \ - M^{n+1} \times ( \Delta M^{n+1} + H_{ext})\,,\\
	      M^{n+1}|_{t=0} =M_0 \in S^{d - 1} \,.
	    \end{array}
	  \right.
	\end{equation}
	By employing Proposition \ref{Prop-Local-LLG}, the Cauchy problem \eqref{MEL-5} has a unique solution $M^{n+1}$ satisfying $\nabla M^{n+1} \in C(0,\widetilde{T}_{n+1}; H^s)$ on the maximal time interval $[0,\widetilde{T}_{n+1})$. We denote by
	\begin{equation*}
	  \begin{aligned}
	    T^*_{n+1} = \min \{ \widehat{T}_{n+1}, \overline{T}_{n+1}, \widetilde{T}_{n+1} \} > 0 \,,
	  \end{aligned}
	\end{equation*}
	and then the proof of Lemma \ref{Lmm-Exist-Appr} is finished.	
\end{proof}

We remark that $T^{*}_{n+1}\leq T^{*}_{n}$.

\subsection{Uniform energy bounds of the iterative approximate system}

The {\em key point} to prove the local well-posedness is to seek a positive lower bound of $T^{*}_{n+1}$ and the uniform energy of the iterative approximate system \eqref{MEL-app}, which will be shown in Lemma \ref{lem-4}. Then, by the compactness argument and Lemma \ref{Lmm-|M|=1}, we can pass to the limits in the system \eqref{MEL-app} and then reach our goal, which is a standard process. We define the following iterative approximate energy functional $E_{n+1} (t)$ and dissipative rate functional $D_{n+1} (t)$
\begin{equation*}
  \begin{aligned}
    E_{n+1} (t) & := \| v^{n+1} \|^2_{H^s} + \| F^{n+1} \|^2_{H^s} + \| \nabla M^{n+1}\|^2_{H^s} \,, \\
   D_{n+1} (t) & := \nu \| \nabla v^{n+1} \|^2_{H^s} + \| \Delta M^{n+1} \|^{2}_{H^s} \,,
  \end{aligned}
\end{equation*}
and precisely state our key lemma.

\begin{lemma}\label{Lmm-Lower-Bnd-T}
Assume that $(v^{n+1}, F^{n+1},M^{n+1}) $ is the solution to the iterative approximate system \eqref{MEL-app} and we define
\begin{equation*}
  T_{n+1} = \sup \Big\{ \tau \in [0,T^{*}_{n+1} ) ; \sup_{t \in [0,\tau]} E_{n+1} (t)+\int_0^\tau D_{n+1} (t) \d t \leq  B \Big\} \geq 0 \,,
\end{equation*}
where $T^{*}_{n+1}> 0$ is the maximal existence time of the iterative approximate system \eqref{MEL-app}. Then for any fixed $ B > \mathcal{E}_0 $, there is a constant $T>0$, depends only on $H_{ext}, \mathcal{E}_0, \nu, s$, such that
$$T_{n+1} > T>0$$
for all $n\geq 0$.
\end{lemma}

\begin{proof}[Proof of Lemma \ref{Lmm-Lower-Bnd-T}]
By the continuity of the iterative approximate energy functional $E_{n+1} (t)$, we known that $T_{n+1}>0$. If the sequence ${T_n;n=1,2,3,...}$ is increasing, then $ T_n \geq T_1 > 0 $ and the conclusion is obviously holds.We thereby only need consider the case that the sequence $\{T_n\}$ is not increasing. Now we choose a strictly increasing sequence $\{n_p\}^{\Lambda}_{p=1}$ as follows:
\begin{equation}\label{Def-np}
  \begin{aligned}
    n_1 = 1, \quad n_{p+1} = \min \big\{ n; n > n_p, T_n  <T_{n_{p}} \big\} \,.
  \end{aligned}
\end{equation}
If $ \Lambda < \infty$, the conclusion also automatically holds. Consequently, we merely need to consider the case $\Lambda = \infty$. By the definition of $\{ n_p \}$ in \eqref{Def-np}, we know that $ \{ T_{n_{p}} \}^\infty_{p=1}$ is strictly decreasing, so that our goal is to prove
$$ \lim_{p \rightarrow \infty} T_{n_{p}}> 0 \,. $$

From employing the almost same arguments and process in deriving the a priori
estimates in Proposition \ref{Prop-AE1}, we deduce that
\begin{equation}\label{eq4}
  \begin{aligned}
    & \tfrac{1}{2} \tfrac{\d}{\d t} \big( \| v^{n+1} \|^{2}_{H^s} + \| F^{n+1} \|^{2}_{H^s} + \| \nabla M^{n+1} \|^{2}_{H^s} \big) + \nu \| \nabla v^{n+1} \|^{2}_{H^s} + \| \Delta M^{n+1} \|^{2}_{H^s}\\
    & \leq C ( \| v^{n}\|_{H^s} \| \nabla v^{n} \|_{H^s} + \| \nabla M^{n} \|^{2}_{H^s} + \| H_{ext} \|_{H^{s+1}} ) \| v^{n+1} \|_{H^s} \\
    & + C \| F^{n} \|^{2}_{H^s} \| \nabla v^{n+1} \|_{H^s}  + C ( \| \nabla v^{n} \|^{2}_{H^s} \| F^{n} \|^{2}_{H^s} ) \| F^{n+1} \|^{2}_{H^s} \\
    & + C ( \| v^{n} \|^{2}_{H^s} + \| \nabla M^{n+1} \|^{2}_{H^s} + \| H_{ext} \|^{2}_{H^s} + \| \nabla M^{n+1} \|^{4}_{H^s} ) \| \nabla M^{n+1} \|^{2}_{H^s} \,,
 \end{aligned}
\end{equation}
which immediately reduces to
\begin{equation}\label{Iter-Uniform-1}
  \begin{aligned}
    \tfrac{\d}{\d t} E_{n+1} (t) + D_{n+1} (t) \leq  C ( 1 + E^{2}_{n} (t) + E^{\frac{1}{2}}_{n} (t) D^{\frac{1}{2}}_{n} (t) ) (1 + E_{n+1} (t) )^3
 \end{aligned}
\end{equation}
for all $t \in [ 0, T_{n+1}^* )$. Here $E_n (t)$ and $D_n (t)$ are both well-defined, since $T_{n+1}^* \leq T_n^*$.

Recalling the definition of $\{ n_p \}_{p=1}^\infty$ in \eqref{Def-np}, we know that for any $N < n_p$
\begin{equation*}
  \begin{aligned}
    T_N > T_{n_p} \,.
  \end{aligned}
\end{equation*}
We take $n = n_p - 1$ in \eqref{Iter-Uniform-1}, and then by the definition of $T_n$ we have
\begin{equation}\label{Iter-Uniform-2}
  \begin{aligned}
    \tfrac{\d}{\d t} E_{n_p} (t) + D_{n_p} (t) \leq \Theta_{n_p - 1} (t) \big[ 1 + E_{n_p} (t) \big]^3 \leq (1 + \mathcal{E}_0 )^2 \Theta_{n_p - 1} (t) \big[ 1 + E_{n_p} (t) \big]
  \end{aligned}
\end{equation}
for all $t \in [ 0, T_{n_p} ]$, where $ \Theta_{n_p - 1} (t) = C \Big( 1 + E^{2}_{n_p - 1} (t) + E^{\frac{1}{2}}_{n_p - 1} (t) D^{\frac{1}{2}}_{n_p - 1} (t) \Big) > 0 $ belongs to $L^1 ([0, T_{n_p}])$. Moreover, from the definition of $T_{n_p-1}$ and the fact $T_{n_p - 1} > T_{n_p}$, we deduce that
\begin{equation}\label{Iter-Uniform-3}
  \begin{aligned}
    \int_0^t \Theta_{n_p - 1} (\tau) \d t \leq C \big[ (1 + B^2) t + B \sqrt{t} \ \big] \leq C ( 1 + B )^2 ( t + \sqrt{t} \ )
  \end{aligned}
\end{equation}
for all $t \in [ 0, T_{n_p} ]$. Noticing that $E_{n_{p}} (0) = \mathcal{E}_0$, we solve the ODE equation \eqref{Iter-Uniform-2} that for all $t \in [ 0, T_{n_p} ]$
\begin{equation*}
  \begin{aligned}
    E_{n_{p}} (t) \leq & - 1 + ( 1 + \mathcal{E}_0 ) \exp \Big[ ( 1 + \mathcal{E}_0 )^2 \int_0^t \Theta_{n_p - 1} (\tau) \d t \Big] \\
    \leq & \underbrace{ - 1 + ( 1 + \mathcal{E}_0 ) \exp \Big[ C ( 1 + \mathcal{E}_0 )^2 ( 1 + B )^2 ( t + \sqrt{t} \ ) \Big] }_{G(t)} \,,
 \end{aligned}
\end{equation*}
where the function $G(t)$ is strictly increasing, continuous on $[0,T_{n_p}]$ and $G(0) = \mathcal{E}_0$. Plugging the above inequality into the ODE inequality \eqref{Iter-Uniform-2} and then integrating on $[0, t]$ for
any $ t \in  [0, T_{n_p} ]$, we estimate that
\begin{equation*}
  \begin{aligned}
    E_{n_{p}} (t) + \int_0^t D_{n_{p}} (\tau) \d \tau \leq & \mathcal{E}_0 + ( 1 + \mathcal{E}_0 )^2 \int_0^t ( 1 + G(\tau) ) \Theta_{n_p - 1} (\tau) \d \tau \\
    \leq &  \mathcal{E}_0 + ( 1 + \mathcal{E}_0 )^2 ( 1 + G(t) ) \int_0^t  \Theta_{n_p - 1} (\tau) \d \tau \\
    \leq & \underbrace{ \mathcal{E}_0 + C ( 1 + \mathcal{E}_0 )^2 ( 1 + B )^2 ( 1 + G(t) ) ( t + \sqrt{t} \ ) }_{H(t)} \,,
  \end{aligned}
\end{equation*}
where we utilize the monotonicity of the function $G(t)$ and the bound \eqref{Iter-Uniform-3}. One notices that the function $H(t)$ is continuous and strictly increasing on $[0,T_{n_p}]$ with $H(0) = \mathcal{E}_0$. Consequently, for any $B > \mathcal{E}_0$, there is a $t^* = t^* (B) > 0$ such that
\begin{equation*}
  \begin{aligned}
    H(t) \leq B
  \end{aligned}
\end{equation*}
holds for all $t \in [0, t^*]$, which immediately yields that
\begin{equation*}
  \begin{aligned}
    E_{n_{p}}(t))+\int_0^t D_{n_{p}}(\tau)d\tau  \leq B
  \end{aligned}
\end{equation*}
for all $t \in [0, t^*]$. By the definition of $T_n$, we derive that $T_{n_{p}}\geq t ^{*}> 0$, hence
$$ T=\lim_{p \rightarrow \infty} T_{n_{p}}\geq t^{*}> 0 \,.$$
Consequently, we complete the proof of Lemma \ref{Lmm-Lower-Bnd-T}.
\end{proof}

\subsection{Proof of Theorem \ref{Thm-1}: local well-posedness}

By Lemma \ref{Lmm-Lower-Bnd-T}, we know that for any fixed $B > \mathcal{E}_0$, there is a $T>0$ such that for all integer $n \geq 0$ and $t \in [0,T]$
 \begin{equation*}
    \begin{aligned}
      \sup_{t \in [0,T]} & \Big{(} \| v^{n+1} \|^2_{H^s} +  \| F ^{n+1} \|^2_{H^s} + \| \nabla M^{n+1} \|^2_{H^s} \Big{)} \\
      & + \int_0^T \Big{(} \nu \| \nabla v^{n+1} \|^2_{H^s} + \| \Delta M ^{n+1} \|^2_{H^s}  \Big{)} \d t \leq B \,.
    \end{aligned}
  \end{equation*}
Then, by compactness arguments and Lemma \ref{Lmm-|M|=1}, we get vector-valued functions $(v, M) \in \R^d \times \mathbb{S}^{d-1}$ and matrix-valued function $F \in \R^{d \times d}$ satisfying $ v, F, \nabla M \in L^\infty (0,T; H^s (\R^d)) $ and $\nabla v , \Delta M \in L^2 (0,T; H^s (\R^d))$, which solve the evolutionary model \eqref{MEL} for magnetoelasticity with the initial conditions \eqref{IC-MEL}. Moreover, $(v, F, M)$ satisfies the bound
 \begin{equation*}
    \begin{aligned}
      \sup_{t \in [0,T]} & \Big{(} \| v \|^2_{H^s} +  \| F \|^2_{H^s} + \| \nabla M \|^2_{H^s} \Big{)} \\
      & + \int_0^T \Big{(} \nu \| \nabla v \|^2_{H^s} + \| \Delta M \|^2_{H^s}  \Big{)} \d t \leq B \,.
    \end{aligned}
  \end{equation*}
Then the proof of Theorem \eqref{Thm-1} is finished.

\section{Global well-posedness with small initial data}\label{Sec-Global}

In this section, we will construct a global classical solution to the magneto-elasticity model \eqref{MEL} under the external magnetic field $H_{ext} = 0$ with small initial data. As stated in Section \ref{Sec-Intro}, we will start from \eqref{MEL-reformulate} and prove a global a priori estimate for the unique solution constructed in Theorem \ref{Thm-1} provided that the initial data satisfies \eqref{IC-small-size}.

One notices that the incompressibility \eqref{Incomprsblt-F} reduces to
\begin{equation}
  \begin{aligned}
    \det ( I + G ) = 1 \,,
  \end{aligned}
\end{equation}
which will play an essential role in deriving the global energy estimate of \eqref{MEL-reformulate}. More precisely,
\begin{equation}
  \begin{aligned}
    1 = \det ( I + G ) = 1 + \mathrm{tr} \, G + O (|G|^2) \,,
  \end{aligned}
\end{equation}
namely, $\mathrm{tr}\, G = O (|G|^2)$, and therefore
\begin{equation}\label{Key-Structure}
  \begin{aligned}
    \| \nabla \cdot \psi \|_{H^s} = \| \mathrm{tr}\, G \|_{H^s} \leq C \| \nabla \psi \|^2_{H^s} \,.
  \end{aligned}
\end{equation}

Before deriving the global energy estimate of \eqref{MEL-reformulate}, let us first recall the following lemma from \cite{Temam-1977-BOOK}:
\begin{lemma}\label{Lmm-General-Stokes}
	Let us suppose that
	\begin{equation*}
	  \begin{aligned}
	    v \in W^{1,\alpha} (\mathbb{R}^d)\,, \ q \in L^\alpha (\mathbb{R}^d) \,, \ 2 \leq \alpha < + \infty
	  \end{aligned}
	\end{equation*}
	are solutions of the generalized Stokes problem
	\begin{equation*}
	  \left\{
	    \begin{array}{c}
	      - \Delta v + \nabla q = f \qquad \textrm{in} \quad \mathbb{R}^d \,, \\[2mm]
	      \nabla \cdot v = g \qquad \textrm{in} \quad \mathbb{R}^d \,.
	    \end{array}
	  \right.
	\end{equation*}
	If $f \in W^{m, \alpha} (\mathbb{R}^d)$ and $g \in W^{m+1, \alpha} (\mathbb{R}^d)$, then $v \in W^{m+2, \alpha} (\mathbb{R}^d)$, $q \in W^{m+1, \alpha} (\mathbb{R}^d)$ and there exists a constant $c_0 ( \alpha, m, d ) > 0$ such that
	\begin{equation*}
	  \begin{aligned}
	    \| v \|_{W^{m+2, \alpha} (\mathbb{R}^d)} + \| q \|_{W^{m+1, \alpha} (\mathbb{R}^d)} \leq c_0 \big( \| f \|_{W^{m, \alpha} (\mathbb{R}^d)} + \| g \|_{W^{m+1, \alpha} (\mathbb{R}^d)} \big) \,.
	  \end{aligned}
	\end{equation*}
\end{lemma}

There is a more general version of Lemma \ref{Lmm-General-Stokes}; see Proposition 2.2 in Chapter I, $\S$ 2 of \cite{Temam-1977-BOOK} for details.

We introduce the energy functional
\begin{equation}\label{Glob-Ener-E}
  \begin{aligned}
    \mathbb{E}_s (t) = \delta^2 \| v \|^2_{H^s} + \| \nabla M \|^2_{H^s} + \delta \| \nabla \psi \|^2_{H^s} + \| \partial_t v \|^2_{H^{s-2}} + \| \nabla \partial_t \psi \|^2_{H^{s-2}}
  \end{aligned}
\end{equation}
and the energy dissipative rate functional
\begin{equation}\label{Glob-Ener-D}
  \begin{aligned}
    \mathbb{D}_s (t) = \tfrac{1}{2} \delta^2 \nu \| \nabla v \|^2_{H^s} + \delta^2 \nu \| \nabla \partial_t \psi \|^2_{H^{s-2}} + 2 \| \Delta M \|^2_{H^s} + \tfrac{\delta}{2 \nu} \| \nabla \psi \|^2_{H^s} + \nu \| \nabla \partial_t v \|^2_{H^{s-2}}
  \end{aligned}
\end{equation}
for some small $\delta > 0$. Actually, the small constant $\delta > 0$ can be explicitly determined in \eqref{Small-delta}. Then we give the following proposition.

\begin{proposition}\label{Prop-Global-Est}
	Let $s \geq 3$ be any integer. Assume that $(v,F,M)$ is the solution on $[ 0, T_{max} )$ to \eqref{MEL}-\eqref{IC-MEL} constructed in Theorem \ref{Thm-1} under $H_{ext} = 0$, where $T_{max}$ is the lifespan of solutions constructed in Theorem \ref{Thm-1}. Let $ (v, G, M) $ be such that $G = F^{-1} - I = \nabla \psi$. Then there is a constant $c_1 (s, d) > 0$ such that
	\begin{equation}\label{Global-Est}
	  \begin{aligned}
	    \tfrac{1}{2} \tfrac{\d}{\d t} \mathbb{E}_s (t) + \mathbb{D}_s (t) \leq c_1 \big(\mathbb{E}_s^\frac{1}{2} (t) + \mathbb{E}_s^\frac{3}{2} ( t ) \big) \mathbb{D}_s (t)
	  \end{aligned}
	\end{equation}
	holds for all $t \in [0, T_{max})$.
\end{proposition}

\begin{proof}[Proof of Proposition \ref{Prop-Global-Est}]
	According to the structure of the energy functional $\mathbb{E}_s (t)$ defined in \eqref{Glob-Ener-E}, we will first estimate the norms $\| v \|^2_{H^s}$, $\| \nabla \psi \|^2_{H^s}$, $ \| \nabla M \|^2_{H^s} $, $\| \partial_t v \|^2_{H^{s-2}}$ and $\| \partial_t \psi \|^2_{H^{s-2}}$ respectively. Finally, we deduce the close energy estimate \eqref{Global-Est} from the bounds of previous norms after adjusting the proper coefficients.
	
	{\em Step 1. Estimate the norm $\| v \|^2_{H^s}$.} For all $|m| \leq s$, we act the derivative operator $\partial^m$ on the first $v$-equation of \eqref{MEL-reformulate}, take inner product with $\partial^m v$ and integrate by parts over $x \in \mathbb{R}^d$. Then we have
	\begin{equation}\label{I1-I2-I3}
	  \begin{aligned}
	    & \tfrac{1}{2} \tfrac{\d}{\d t} \| \partial^m v \|^2_{L^2} + \nu \| \nabla \partial^m v \|^2_{L^2} - \big\langle \nabla \partial^m \psi , \nabla \partial^m v \big\rangle \\
	    = & \underset{I_1}{\underbrace{ - \big\langle \partial^m (v \cdot \nabla v), \partial^m v \big\rangle }} + \underset{I_2}{\underbrace{ \big\langle \partial^m g (G) , \nabla \partial^m v \big\rangle }} \ \underset{I_3}{\underbrace{ - \big\langle \partial^m \nabla \cdot ( \nabla M \odot \nabla M ), \partial^m v \big\rangle }} \,.
	  \end{aligned}
	\end{equation}
	For the term $I_1$, we estimate that
	\begin{equation}\label{I1}
	  \begin{aligned}
	    I_1 = & - \sum_{0 \neq m' \leq m} C_m^{m'} \big\langle \partial^{m'} v \cdot \nabla \partial^{m-m'} v , \partial^m v \big\rangle \\
	    \leq & C \sum_{0 \neq m' \leq m} \| \partial^{m'} v \|_{L^4} \| \nabla \partial^{m-m'} v \|_{L^4} \| \partial^m v \|_{L^2} \\
	    \leq & C \sum_{0 \neq m' \leq m} \| \partial^{m'} v \|_{H^1} \| \nabla \partial^{m-m'} v \|_{H^1} \| \partial^m v \|_{L^2} \\
	    \leq & C \| v \|_{H^s} \| \nabla v \|^2_{H^s} \leq C (\delta) \mathbb{E}^\frac{1}{2}_s (t) \mathbb{D}_s (t) \,,
	  \end{aligned}
	\end{equation}
	where we utilize the divergence-free property of $v$ and the Sobolev embedding $H^1 (\mathbb{R}^d) \hookrightarrow L^4 (\mathbb{R}^d)$ for $d=2,3$. Since $g (G) = O (|G|^2)$ and $G = \nabla \psi$, we have
	\begin{equation}\label{I2}
	  \begin{aligned}
	    I_2 \leq \| \partial^m g (G) \|_{L^2} \| \nabla \partial^m v \|_{L^2} \leq C \| \nabla \psi \|^2_{H^s} \| \nabla v \|_{H^s} \leq C(\delta) \mathbb{E}^\frac{1}{2}_s (t) \mathbb{D}_s (t) \,.
	  \end{aligned}
	\end{equation}
	Furthermore, from the Sobolev interpolation inequality $ \| f \|_{L^4(\mathbb{R}^d)} \leq C \| f \|_{L^2 (\mathbb{R}^d)}^{1 - \tfrac{d}{4}} \| \nabla f \|^\frac{d}{4}_{L^2 (\mathbb{R}^d)} $ for $d=2,3$, one easily derives that
	\begin{equation}\label{I3}
	  \begin{aligned}
	    I_3 = & \sum_{m' \leq m} C_m^{m'} \big\langle \nabla \partial^{m'} M \odot \nabla \partial^{m-m'} M , \nabla \partial^m v \big\rangle \\
	    \leq & C \sum_{m' \leq m} \| \nabla \partial^{m'} M \|_{L^4} \| \nabla \partial^{m-m'} M \|_{L^4} \| \nabla \partial^m v \|_{L^2} \\
	    \leq & C \sum_{m' \leq m} \| \nabla \partial^{m'} M \|_{L^2}^{1 - \frac{d}{4}} \| \Delta \partial^{m'} M \|_{L^2}^\frac{d}{4} \| \nabla \partial^{m- m'} M \|_{L^2}^{1 - \frac{d}{4}} \| \Delta \partial^{m-  m'} M \|_{L^2}^\frac{d}{4} \| \nabla \partial^m v \|_{L^2} \\
	    \leq & C \| \nabla M \|_{H^s} \| \Delta M \|_{H^s} \| \nabla v \|_{H^s} \leq C(\delta) \mathbb{E}^\frac{1}{2}_s (t) \mathbb{D}_s (t) \,.
	  \end{aligned}
	\end{equation}
	Then, substituting the inequalities \eqref{I1}, \eqref{I2} and \eqref{I3} into the equality \eqref{I1-I2-I3}, summing up for $|m| \leq s$ and the bound $\sum_{|m| \leq s} \big\langle \nabla \partial^m \psi , \nabla \partial^m v \big\rangle \leq \| \nabla \psi \|_{H^s} \| \nabla v \|_{H^s} $ give us
	\begin{equation}\label{norm-v-Hs}
	  \begin{aligned}
	    \tfrac{1}{2} \tfrac{\d}{\d t} \| v \|^2_{H^s} + \nu \| \nabla v \|^2_{H^s} - \| \nabla \psi \|_{H^s} \| \nabla v \|_{H^s} \leq C(\delta) \mathbb{E}^\frac{1}{2}_s (t) \mathbb{D}_s (t) \,.
	  \end{aligned}
	\end{equation}
	
	{\em Step 2. Estimate the norm $\| \nabla \psi \|^2_{H^s}$.} Via acting the derivative operator $\nabla \partial^m$ on the third $\psi$-equation of \eqref{MEL-reformulate} for all $|m| \leq s$ and taking inner product by dot with $\nabla \partial^m \psi$, we can obtain
	\begin{equation}\label{II1-II2}
	  \begin{aligned}
	    \tfrac{1}{2} \tfrac{\d}{\d t} \| \nabla \partial^m \psi \|^2_{L^2} + \tfrac{1}{\nu} \| \nabla \partial^m \psi \|^2_{L^2} = \underset{I\!I_1}{\underbrace{ - \tfrac{1}{\nu} \big\langle \nabla \partial^m w , \nabla \partial^m \psi \big\rangle }} \  \underset{I\!I_2}{\underbrace{ - \big\langle \nabla \partial^m ( v \cdot \nabla \psi ) , \nabla \partial^m \psi \big\rangle }} \,,
	  \end{aligned}
	\end{equation}
	where $w = \nu v - \psi $. From the first $v$-equation of \eqref{MEL-reformulate}, we deduce that $\partial_i w$ $(i = 1,2, \cdots , d)$ obeys the following generalized Stokes system
	\begin{equation}\no
	  \left\{
	    \begin{array}{l}
	      - \Delta \partial_i w + \nabla \partial_i q = - \partial_i \partial_t v - \partial_i (v \cdot \nabla v) + \partial_i \nabla \cdot \big( g (G) - \nabla M \odot \nabla M \big) \,, \\
	      \nabla \cdot \partial_i w = - \partial_i \nabla \cdot \psi \,,
	    \end{array}
	  \right.
	\end{equation}
	where the divergence-free property of $v$ is utilized. Then Lemma \ref{Lmm-General-Stokes} implies that for all $i = 1,2, \cdots, d$
	\begin{equation}\label{w}
	  \begin{aligned}
	    \| \partial_i w \|_{H^s} + \| \partial_i q \|_{H^{s-1}} \leq c_0 \big( \| \partial_i \nabla \cdot \psi \|_{H^{s-1}} + \| \partial_i \partial_t v \|_{H^{s-2}} \\
	    + \| \partial_i (v \cdot \nabla v) \|_{H^{s-2}} + \| \partial_i \nabla \cdot \big( g(G) - \nabla M \odot \nabla M \big) \|_{H^{s-2}} \big) \,.
	  \end{aligned}
	\end{equation}
	By the key relation \eqref{Key-Structure}, we imply that
	\begin{equation}\label{w1}
	  \begin{aligned}
	    \| \partial_i \nabla \cdot \psi \|_{H^{s-1}} = \| \partial_i \mathrm{tr}\, G \|_{H^{s-1}} \leq C \| \nabla \psi \|^2_{H^s}
	  \end{aligned}
	\end{equation}
	For the quantity $ \| \partial_i (v \cdot \nabla v) \|_{H^{s-2}} $, one can deduce from the Sobolev embedding $H^1 (\mathbb{R}^d) \hookrightarrow L^4 (\mathbb{R}^d)$ for $d=2,3$ that
	\begin{equation}\label{w2}
	  \begin{aligned}
	    & \| \partial_i (v \cdot \nabla v) \|^2_{H^{s-2}} = \sum_{|m| \leq s-2} \| \partial^m \partial_i (v \cdot \nabla v) \|^2_{L^2} \\
	    \leq & \sum_{|m| \leq s-1} \sum_{m' \leq m} C_m^{m'} \| \partial^{m'} v \cdot \nabla \partial^{m-m'} v \|^2_{L^2} \\
	    \leq & \sum_{|m| \leq s-1} \sum_{m' \leq m} C_m^{m'} \| \partial^{m'} v \|^2_{L^4} \| \nabla \partial^{m-m'} v \|^2_{L^4} \\
	    \leq & C \| v \|^2_{H^s} \| \nabla v \|^2_{H^s} \,.
	  \end{aligned}
	\end{equation}
	Similarly, we can also derive that
	\begin{equation}\label{w3}
	  \begin{aligned}
	    & \| \partial_i \nabla \cdot ( \nabla M \odot \nabla M ) \|^2_{H^{s-2}} = \| \partial_i ( \nabla^2 M \odot \nabla M + \nabla M \odot \Delta M ) \|^2_{H^{s-2}} \\
	    \leq & C \sum_{|m| \leq s - 1} \sum_{m' \leq m} \| \Delta \partial^{m'} M \|^2_{L^4} \| \nabla \partial^{m-m'} M \|^2_{L^4} \\
	    \leq & C \sum_{|m| \leq s - 1} \sum_{m' \leq m} \| \Delta \partial^{m'} M \|^2_{H^1} \| \nabla \partial^{m-m'} M \|^2_{H^1} \\
	    \leq & C \| \Delta M \|^2_{H^s} \| \nabla M \|^2_{H^s} \,.
	  \end{aligned}
	\end{equation}
	Recalling that $g (G) = O(|G|^2)$ and $G = \nabla \psi$, we have $\| g (G) \|_{H^s} \leq C  \| \nabla \psi \|^2_{H^s}$, which immediately implies that
	\begin{equation}\label{w4}
	  \begin{aligned}
	    \| \partial_i \nabla \cdot g (G) \|_{H^{s-2}} \leq \| g (G) \|_{H^s} \leq C  \| \nabla \psi \|^2_{H^s} \,.
	  \end{aligned}
	\end{equation}
	Consequently, by plugging the bounds \eqref{w1}, \eqref{w2}, \eqref{w3} and \eqref{w4} into the relation \eqref{w}, we easily obtain
	\begin{equation}\label{nabla-w}
	  \begin{aligned}
	    \| \nabla w \|_{H^s} + \| \nabla q \|_{H^{s-1}} \leq & 4 c_0 \| \nabla \partial_t v \|_{H^{s-2}} + C ( \| v \|_{H^s} + \| \nabla \psi \|_{H^s} + \| \nabla M \|_{H^s} ) \\
	    & \qquad \qquad \qquad \quad \times ( \| \nabla v \|_{H^s} + \| \nabla \psi \|_{H^s} + \| \Delta M \|_{H^s} ) \,.
	  \end{aligned}
	\end{equation}
	Thus, by the previous bound \eqref{nabla-w}, the H\"older inequality and the definitions of $\mathbb{E}_s (t)$, $\mathbb{D}_s (t)$ in \eqref{Glob-Ener-E}, \eqref{Glob-Ener-D} respectively, we imply that
	\begin{equation}\label{II1}
	  \begin{aligned}
	    I\!I_1 \leq & \tfrac{1}{\nu} \| \nabla w \|_{H^s} \| \nabla \psi \|_{H^s} \\
	    \leq & \tfrac{4 c_0}{\nu} \| \nabla \partial_t v \|_{H^{s-2}} \| \nabla \psi \|_{H^s} + \tfrac{C}{\nu} ( \| v \|_{H^s} + \| \nabla \psi \|_{H^s} + \| \nabla M \|_{H^s} ) \\
	    & \qquad \qquad \qquad \times ( \| \nabla v \|_{H^s} + \| \nabla \psi \|_{H^s} + \| \Delta M \|_{H^s} ) \| \nabla \psi \|_{H^s} \\
	    \leq & \tfrac{4 c_0}{\nu} \| \nabla \partial_t v \|_{H^{s-2}} \| \nabla \psi \|_{H^s} + C (\delta) \mathbb{E}_s^\frac{1}{2} (t) \mathbb{D}_s (t) \,.
	  \end{aligned}
	\end{equation}
	
	Noticing that $\nabla \cdot v = 0$, we deduce that
	\begin{equation}\label{II2-Equ}
	  \begin{aligned}
	    I\!I_2 = & - \big\langle \partial^m ( \nabla v \cdot \nabla \psi ) + \partial^m ( v \cdot \nabla \nabla \psi ) , \nabla \partial^m \psi \big\rangle \\
	    = & \underset{I\!I_{21}}{\underbrace{ - \big\langle \nabla \partial^m v \cdot \nabla \psi , \nabla \partial^m \psi \big\rangle }} \  \underset{I\!I_{22}}{\underbrace{ - \big\langle \nabla v \cdot \nabla \partial^m \psi , \nabla \partial^m \psi \big\rangle }} \\
	    & \underset{I\!I_{23}}{\underbrace{ - \sum_{0 \neq m' < m} C_m^{m'} \big\langle \nabla \partial^{m-m'} v \cdot \nabla \partial^{m'} \psi , \nabla \partial^m \psi \big\rangle }} \\
	    & \underset{I\!I_{24}}{\underbrace{ - \sum_{|m'| = 1} C_m^{m'} \big\langle \partial^{m'} v \cdot \nabla \partial^{m-m'} \nabla \psi , \nabla \partial^m \psi \big\rangle }} \\
	    & \underset{I\!I_{25}}{\underbrace{ - \sum_{m' \leq m, |m'| \geq 2} C_m^{m'} \big\langle \partial^{m'} v \cdot \nabla \partial^{m-m'} \nabla \psi , \nabla \partial^m \psi \big\rangle }} \,.
	  \end{aligned}
	\end{equation}
	By the Sobolev embedding $H^2 (\mathbb{R}^d) \hookrightarrow L^\infty (\mathbb{R}^d)$ for $d = 2,3$, we estimate that
	\begin{equation}\label{II21}
	  \begin{aligned}
	    I\!I_{21} \leq & \| \nabla \partial^m v \|_{L^2} \| \nabla \psi \|_{L^\infty} \| \nabla \partial^m \psi \|_{L^2} \\
	    \leq & C \| \nabla \partial^m v \|_{L^2} \| \nabla \psi \|_{H^2} \| \nabla \partial^m \psi \|_{L^2} \\
	    \leq & C \| \nabla v \|_{H^s} \| \nabla \psi \|^2_{H^s} \,,
	  \end{aligned}
	\end{equation}
	and
	\begin{equation}\label{II22}
	  \begin{aligned}
	    I\!I_{22} \leq & \| \nabla v \|_{L^\infty} \| \nabla \partial^m \psi \|^2_{L^2} \leq C \| \nabla v \|_{H^2} \| \nabla \partial^m \psi \|^2_{L^2} \leq C \| \nabla v \|_{H^s} \| \nabla \psi \|^2_{H^s} \,,
	  \end{aligned}
	\end{equation}
	and
	\begin{equation}\label{II24}
	  \begin{aligned}
	    I\!I_{24} \leq & \sum_{|m'| = 1} C_m^{m'} \| \partial^{m'} v \|_{L^\infty} \| \nabla \partial^{m-m'} \nabla \psi \|_{L^2} \| \nabla \partial^m \psi \|_{L^2} \\
	    \leq & C \sum_{|m'| = 1} \| \partial^{m'} v \|_{H^2} \| \nabla \partial^{m-m'} \nabla \psi \|_{L^2} \| \nabla \partial^m \psi \|_{L^2} \\
	    \leq & C \| \nabla v \|_{H^s} \| \nabla \psi \|^2_{H^s} \,.
	  \end{aligned}
	\end{equation}
	Furthermore, from the Sobolev embedding $H^1 (\mathbb{R}^d) \hookrightarrow L^4 (\mathbb{R}^d)$ for $d = 2,3$, we derive that
	\begin{equation}\label{II23}
	  \begin{aligned}
	    I\!I_{23} \leq &  \sum_{0 \neq m' < m } C_m^{m'} \| \nabla \partial^{m-m'} v \|_{L^4} \| \nabla \partial^{m'} \psi \|_{L^4} \| \nabla \partial^m \psi \|_{L^2} \\
	    \leq & C \sum_{0 \neq m' < m } \| \nabla \partial^{m-m'} v \|_{H^1} \| \nabla \partial^{m'} \psi \|_{H^1} \| \nabla \partial^m \psi \|_{L^2} \\
	    \leq & C \| \nabla v \|_{H^s} \| \nabla \psi \|^2_{H^s} \,,
	  \end{aligned}
	\end{equation}
	and
	\begin{equation}\label{II25}
	  \begin{aligned}
	    I\!I_{25} \leq & C \sum_{m' \leq m , |m'| \geq 2} \| \partial^{m'} v \|_{L^4} \| \nabla \partial^{m-m'} \nabla \psi \|_{L^4} \| \nabla \partial^m \psi \|_{L^2} \\
	    \leq & C \sum_{m' \leq m , |m'| \geq 2} \| \partial^{m'} v \|_{H^1} \| \nabla \partial^{m-m'} \nabla \psi \|_{H^1} \| \nabla \partial^m \psi \|_{L^2} \\
	    \leq & C \| \nabla v \|_{H^s} \| \nabla \psi \|^2_{H^s} \,.
	  \end{aligned}
	\end{equation}
	As a result, it is derived from plugging the bounds \eqref{II21}, \eqref{II22}, \eqref{II24}, \eqref{II23} and \eqref{II25} into the equality \eqref{II2-Equ} that
	\begin{equation}\label{II2}
	  \begin{aligned}
	    I\!I_2 \leq C \| \nabla v \|_{H^s} \| \nabla \psi \|^2_{H^s} \leq C(\delta) \mathbb{E}_s^\frac{1}{2} (t) \mathbb{D}_s (t) \,.
	  \end{aligned}
	\end{equation}
	We now substitute the bounds \eqref{II1} and \eqref{II2} into the equality \eqref{II1-II2} and sum up for $|m| \leq s$. Then we obtain
	\begin{equation}\label{norm-nabla-psi-Hs}
	  \begin{aligned}
	    \tfrac{1}{2} \tfrac{\d}{\d t} \| \nabla \psi \|^2_{H^s} + \tfrac{1}{\nu} \| \nabla \psi \|^2_{H^s} - \tfrac{4 c_0 K_s}{\nu} \| \nabla \partial_t v \|_{H^{s-2}} \| \nabla \psi \|_{H^s} \leq C(\delta) \mathbb{E}_s^\frac{1}{2} (t) \mathbb{D}_s (t) \,,
	  \end{aligned}
	\end{equation}
	where $K_s > 0$ denotes the number of all possible $m \in \mathbb{N}^d$ such that $|m| \leq s$.
	
	{\em Step 3. Estimate the norm $\| \nabla M \|^2_{H^s}$.} For all multi-indexes $ m \in \mathbb{N}^d$ with $|m| \leq s$, we act the derivative operator $\partial^m$ on the last second $M$-equation of \eqref{MEL-reformulate}, take inner product by dot with $\Delta \partial^m M$ and integrate by parts over $x \in \mathbb{R}^d$. Then we have
	\begin{equation}\label{III1-III2-III3}
	  \begin{aligned}
	    & \tfrac{1}{2} \tfrac{\d}{\d t} \| \nabla \partial^m M \|^2_{L^2} + \| \Delta \partial^m M \|^2_{L^2} \\
	    = & \underset{I\!I\!I_1}{\underbrace{ \big\langle \partial^m ( v \cdot \nabla M ) , \Delta \partial^m M \big\rangle }} \ \underset{I\!I\!I_2}{\underbrace{ - \big\langle \partial^m ( |\nabla M|^2 M ) , \Delta \partial^m M \big\rangle }} + \underset{I\!I\!I_3}{\underbrace{ \big\langle \partial^m ( M \times \Delta M ) , \Delta \partial^m M \big\rangle }} \,.
	  \end{aligned}
	\end{equation}
	We decompose the term $I\!I\!I_1$ as
	\begin{equation}\no
	  \begin{aligned}
	    I\!I\!I_1 = \underset{I\!I\!I_{11}}{\underbrace{ \big\langle v \cdot \nabla \partial^m M , \Delta \partial^m M \big\rangle }} + \underset{I\!I\!I_{12}}{\underbrace{ \sum_{0 \neq m' \leq m} C_m^{m'} \big\langle \partial^{m'} v \cdot \nabla \partial^{m-m'} M , \Delta \partial^m M \big\rangle }} \,.
	  \end{aligned}
	\end{equation}
	For the term $I\!I\!I_{11}$, from the Sobolev interpolation $ \| f \|_{L^4 (\mathbb{R}^d)} \leq C \| f \|_{L^2 (\mathbb{R}^d)}^{1 - \tfrac{d}{4}} \| \nabla f \|_{L^2 (\mathbb{R}^d)}^\frac{d}{4} $ with $d=2,3$, we deduce that
	\begin{equation}\label{III11}
	  \begin{aligned}
	    I\!I\!I_{11} \leq & \| v \|_{L^4} \| \nabla \partial^m M \|_{L^4} \| \Delta \partial^m M \|_{L^2} \\
	    \leq & C \| v \|^{1- \frac{d}{4}}_{L^2} \| \nabla v \|_{L^2}^\frac{d}{4} \| \nabla \partial^m M \|^{1- \frac{d}{4}}_{L^2} \| \Delta \partial^m M \|_{L^2}^{1 + \frac{d}{4}} \\
	    = & C \| v \|^{1- \frac{d}{4}}_{L^2} \| \nabla v \|_{L^2}^{\frac{d}{2} - 1} \| \nabla \partial^m M \|^{1- \frac{d}{4}}_{L^2} \| \nabla v \|_{L^2}^{1 - \frac{d}{4}} \| \Delta \partial^m M \|_{L^2}^{1 + \frac{d}{4}} \\
	    \leq & C ( \| v \|_{H^s} + \| \nabla M \|_{H^s} ) ( \| \nabla v \|^2_{H^s} + \| \Delta M \|^2_{H^s} ) \\
	    \leq & C(\delta) \mathbb{E}_s^\frac{1}{2} (t) \mathbb{D}_s (t) \,.
	  \end{aligned}
	\end{equation}
	Furthermore, by the Sobolev embedding $H^1 (\mathbb{R}^d) \hookrightarrow L^4 (\mathbb{R}^d)$ with $d=2,3$, we have
	\begin{equation}\label{III12}
	  \begin{aligned}
	    I\!I\!I_{12} \leq & \sum_{0 \neq m' \leq m} C_m^{m'} \| \partial^{m'} v \|_{L^4} \| \nabla \partial^{m-m'} M \|_{L^4} \| \Delta \partial^m M \|_{L^2} \\
	    \leq & C \sum_{0 \neq m' \leq m} \| \partial^{m'} v \|_{H^1} \| \nabla \partial^{m-m'} M \|_{H^1} \| \Delta \partial^m M \|_{L^2} \\
	    \leq & C \| \nabla v \|_{H^s} \| \nabla M \|_{H^s} \| \Delta M \|_{H^s} \leq C(\delta) \mathbb{E}_s^\frac{1}{2} (t) \mathbb{D}_s (t) \,.
	  \end{aligned}
	\end{equation}
	We thereby know that by combining \eqref{III11} with \eqref{III12}
	\begin{equation}\label{III1}
	  \begin{aligned}
	    I\!I\!I_1 = I\!I\!I_{11} + I\!I\!I_{12} \leq C(\delta) \mathbb{E}_s^\frac{1}{2} (t) \mathbb{D}_s (t) \,.
	  \end{aligned}
	\end{equation}
	We expand the term $I\!I\!I_2$ as
	\begin{equation}\no
	  \begin{aligned}
	    I\!I\!I_2 = & \underset{I\!I\!I_{21}}{\underbrace{  - \sum_{m' \leq m} C_m^{m'} \big\langle ( \nabla \partial^{m'} M \cdot \nabla \partial^{m-m'} M ) M , \Delta \partial^m M \big\rangle }} \\
	    & \underset{I\!I\!I_{22}}{\underbrace{  - \sum_{0 \neq m' \leq m} \sum_{m'' \leq m-m'} C_m^{m'} C_{m-m'}^{m''} \big\langle ( \nabla \partial^{m-m'-m''} M \cdot \nabla \partial^{m''} M ) \partial^{m'} M , \Delta \partial^m M \big\rangle }} \,.
	  \end{aligned}
	\end{equation}
	Since $|M| = 1$, we estimate
	\begin{equation}\label{III21}
	  \begin{aligned}
	    & I\!I\!I_{21} \leq C \sum_{m' \leq m} \| \nabla \partial^{m'} M \|_{L^4} \| \nabla \partial^{m-m'} M \|_{L^4} \| \Delta \partial^m M \|_{L^2} \\
	    \leq & C \sum_{m' \leq m} \| \nabla \partial^{m'} M \|^{1-\frac{d}{4}}_{L^2} \| \Delta \partial^{m'} M \|^\frac{d}{4}_{L^2} \| \nabla \partial^{m-m'} M \|^{1-\frac{d}{4}}_{L^2} \| \Delta \partial^{m-m'} M \|^\frac{d}{4}_{L^2} \| \Delta \partial^m M \|_{L^2} \\
	    \leq & C \| \nabla M \|_{H^s} \| \Delta M \|^2_{H^s} \leq C(\delta) \mathbb{E}_s^\frac{1}{2} (t) \mathbb{D}_s (t) \,,
	  \end{aligned}
	\end{equation}
	where we make use of the Sobolev interpolation $ \| f \|_{L^4 (\mathbb{R}^d)} \leq C \| f \|_{L^2 (\mathbb{R}^d)}^{1 - \tfrac{d}{4}} \| \nabla f \|_{L^2 (\mathbb{R}^d)}^\frac{d}{4} $ with $d=2,3$. By the previous Sobolev interpolation inequality and the Sobolev embedding $H^2 (\mathbb{R}^d) \hookrightarrow L^\infty (\mathbb{R}^d)$ with $d = 2,3$, we yield that
	\begin{equation}\label{III22}
	  \begin{aligned}
	    I\!I\!I_{22} \leq & C \sum_{\substack{m' \leq m , |m'| = 1 \\ m'' \leq m -m'}} \| \nabla \partial^{m-m'-m''} M \|_{L^4} \| \nabla \partial^{m''} M \|_{L^4} \| \partial^{m'} M \|_{L^\infty} \| \Delta \partial^m M \|_{L^2} \\
	    + & C \sum_{\substack{m' \leq m , |m'| \geq 2 \\ m'' \leq m -m'}} \| \nabla \partial^{m-m'-m''} M \|_{L^\infty} \| \nabla \partial^{m''} M \|_{L^\infty} \| \partial^{m'} M \|_{L^2} \| \Delta \partial^m M \|_{L^2} \\
	    \leq & C \sum_{\substack{m' \leq m , |m'| = 1 \,, m'' \leq m -m'}} \| \nabla \partial^{m-m'-m''} M \|_{L^2}^{1 - \tfrac{d}{4}} \| \Delta \partial^{m-m'-m''} M \|_{L^2}^\frac{d}{4} \| \partial^{m'} M \|_{H^2} \\
	    & \qquad \qquad \qquad \qquad \qquad \times \| \nabla \partial^{m''} M \|_{L^2}^{1 - \tfrac{d}{4}} \| \Delta \partial^{m''} M \|_{L^2}^\frac{d}{4} \| \Delta \partial^m M \|_{L^2} \\[2mm]
	    + & C \sum_{\substack{m' \leq m , |m'| \geq 2 \,, m'' \leq m -m'}} \| \nabla \partial^{m-m'-m''} M \|_{H^2} \| \nabla \partial^{m''} M \|_{H^2} \| \Delta M \|_{L^2} \| \Delta \partial^m M \|_{L^2} \\
	    \leq & C \| \nabla M \|_{H^s} \| \Delta M \|^2_{H^s} \leq C(\delta) \mathbb{E}_s^\frac{1}{2} (t) \mathbb{D}_s (t) \,.
	  \end{aligned}
	\end{equation}
	Thus, the bounds \eqref{III21} and \eqref{III22} tell us
	\begin{equation}\label{III2}
	  \begin{aligned}
	    I\!I\!I_2 = I\!I\!I_{21} + I\!I\!I_{22} \leq C(\delta) \mathbb{E}_s^\frac{1}{2} (t) \mathbb{D}_s (t) \,.
	  \end{aligned}
	\end{equation}
	Noticing that $ (M \times \Delta \partial^m M ) \cdot \Delta \partial^m M = 0$, we have
	\begin{equation}\label{III3}
	  \begin{aligned}
	    I\!I\!I_3 = & \sum_{0 \neq m' \leq m} C_m^{m'} \big\langle \partial^{m'} M \times \Delta \partial^{m-m'} M , \Delta \partial^m M \big\rangle \\
	    \leq & C \sum_{0 \neq m' \leq m} \| \partial^{m'} M \|_{L^4} \| \Delta \partial^{m-m'} M \|_{L^4} \| \Delta \partial^m M \|_{L^2} \\
	    \leq & C \sum_{0 \neq m' \leq m} \| \partial^{m'} M \|_{H^1} \| \Delta \partial^{m-m'} M \|_{H^1} \| \Delta \partial^m M \|_{L^2} \\
	    \leq & C \| \nabla M \|_{H^s} \| \Delta M \|^2_{H^s} \leq C(\delta) \mathbb{E}_s^\frac{1}{2} (t) \mathbb{D}_s (t) \,,
	  \end{aligned}
	\end{equation}
	where the Sobolev embedding $H^1 (\mathbb{R}^d) \hookrightarrow L^4 (\mathbb{R}^d)$ for $d = 2,3$ is utilized. We thereby derive from plugging bounds \eqref{III1}, \eqref{III2}, \eqref{III3} into the equality \eqref{III1-III2-III3} and summing up for all $|m| \leq s$ that
	\begin{equation}\label{norm-nabla-M-Hs}
	  \begin{aligned}
	    \tfrac{1}{2} \tfrac{\d}{\d t} \| \nabla M \|^2_{H^s} + \| \Delta M \|^2_{H^s} \leq C(\delta) \mathbb{E}_s^\frac{1}{2} (t) \mathbb{D}_s (t) \,.
	  \end{aligned}
	\end{equation}
	
	{\em Step 4. Estimate the norm $\| \partial_t  v \|^2_{H^{s-2}}$.} For all $|m| \leq s - 2$, we apply the derivative operator $\partial^m \partial_t$ on the first $v$-equation of \eqref{MEL-reformulate}, take $L^2$-inner product by multiplying $\partial^m \partial_t v$ and integrate by parts over $x \in \mathbb{R}^d$. We thereby deduce that
	\begin{equation}\label{IV1-IV2-IV3}
	  \begin{aligned}
	    & \tfrac{1}{2} \tfrac{\d}{\d t} \| \partial^m \partial_t v \|^2_{L^2} + \nu \| \nabla \partial^m \partial_t v \|^2_{L^2} + \big\langle \Delta \partial^m \partial_t \psi , \partial^m \partial_t v \big\rangle \\
	    = & \underset{I\!V_1}{\underbrace{ - \big\langle \partial^m \partial_t ( v \cdot \nabla v ) , \partial^m \partial_t v \big\rangle }} \ \underset{I\!V_2}{\underbrace{ - \big\langle \partial^m \partial_t g (G) , \nabla \partial^m \partial_t v \big\rangle }} \ \underset{I\!V_3}{\underbrace{ - \big\langle \partial^m \partial_t ( \nabla M \odot \nabla M ) , \nabla \partial^m \partial_t v \big\rangle }} \,,
	  \end{aligned}
	\end{equation}
	where the divergence-free property of $\partial^m \partial_t v$ is also used.
	
	The term $I\!V_1$ can be calculated that
	\begin{equation}\no
	  \begin{aligned}
	    I\!V_1 = & \underset{I\!V_{11}}{\underbrace{ - \sum_{m' \leq m} C_m^{m'} \big\langle \partial^{m'} \partial_t \cdot \nabla \partial^{m-m'} v , \partial^m \partial_t v \big\rangle }} \\
	    & \underset{I\!V_{12}}{\underbrace{ - \sum_{0 \neq m' \leq m} C_m^{m'} \big\langle \partial^{m'} v \cdot \nabla \partial^{m-m'} \partial_t v , \partial^m \partial_t v \big\rangle }} \,.
	  \end{aligned}
	\end{equation}
	From the Sobolev interpolation inequality $ \| f \|_{L^4 (\mathbb{R}^d)} \leq C \| f \|_{L^2 (\mathbb{R}^d)}^{1 - \tfrac{d}{4}} \| \nabla f \|_{L^2 (\mathbb{R}^d)}^\frac{d}{4} $ with $d=2,3$ and the H\"older inequality, we deduce that
	\begin{equation}\label{IV11}
	  \begin{aligned}
	    I\!V_{11} \leq & C \sum_{m' \leq m} \| \partial^{m'} \partial_t v \|_{L^4} \| \nabla \partial^{m-m'} v \|_{L^2} \| \partial^m \partial_t v \|_{L^4} \\
	    \leq & C \sum_{m' \leq m} \| \partial^{m'} \partial_t v \|^{1-\tfrac{d}{4}}_{L^2} \| \partial^m \partial_t v \|^{1-\tfrac{d}{4}}_{L^2} \| \nabla \partial^{m'} \partial_t v \|^\frac{N}{d}_{L^2} \| \nabla \partial^m \partial_t v \|^\frac{d}{4}_{L^2} \| \nabla \partial^{m-m'} v \|_{L^2} \\
	    = & C \sum_{m' \leq m} \| \partial^{m'} \partial_t v \|^{1-\tfrac{d}{4}}_{L^2} \| \partial^m \partial_t v \|^{1-\tfrac{d}{4}}_{L^2} \| \nabla \partial^{m-m'} v \|^{\tfrac{d}{2} - 1}_{L^2} \\
	    & \qquad \qquad \times \| \nabla \partial^{m-m'} v \|^{2- \tfrac{d}{2}}_{L^2} \| \nabla \partial^{m'} \partial_t v \|^\frac{d}{4}_{L^2} \| \nabla \partial^m \partial_t v \|^\frac{d}{4}_{L^2} \\
	    \leq & C \| \partial_t v \|^{2-\tfrac{d}{2}}_{H^{s-2}} \| v \|^{\tfrac{d}{2}-1}_{H^s} \| \nabla v \|_{H^s}^{2 - \tfrac{d}{2}} \| \nabla \partial_t v \|_{H^{s-2}}^\frac{d}{2} \\
	    \leq & C(\delta) \mathbb{E}_s^\frac{1}{2} (t) \mathbb{D}_s (t) \,.
	  \end{aligned}
	\end{equation}
	Similarly in \eqref{IV11}, we estimate that
	\begin{equation}\label{IV12}
	  \begin{aligned}
	    I\!V_{12} \leq & C \sum_{0 \neq m' \leq m} \| \partial^{m'} v \|_{L^4} \| \partial^m \partial_t v \|_{L^4} \| \nabla \partial^{m-m'} \partial_t v \|_{L^2} \\
	    \leq & C \sum_{0 \neq m' \leq m} \| \partial^{m'} v \|_{L^2}^{1-\tfrac{d}{4}} \| \nabla \partial^{m'} v \|_{L^2}^\frac{d}{4} \| \partial^m \partial_t v \|_{L^2}^{1-\frac{d}{4}} \| \nabla \partial^m \partial_t v \|_{L^2}^\frac{d}{4} \| \nabla \partial^{m-m'} \partial_t v \|_{L^2} \\
	    \leq & C \| v \|_{H^s}^\frac{d}{4} \| \partial_t v \|_{H^{s-2}}^{1-\frac{d}{4}} \| \nabla v \|_{H^s}^{1 - \frac{d}{4}} \| \nabla \partial_t v \|_{H^{s-2}}^{1+\frac{d}{4}} \\
	    \leq & C(\delta) \mathbb{E}_s^\frac{1}{2} (t) \mathbb{D}_s (t) \,.
	  \end{aligned}
	\end{equation}
	Thus, combining the bounds \eqref{IV11} and \eqref{IV12}, we obtain
	\begin{equation}\label{IV1}
	  \begin{aligned}
	    I\!V_1 = I\!V_{11} + I\!V_{12} \leq C(\delta) \mathbb{E}_s^\frac{1}{2} (t) \mathbb{D}_s (t) \,.
	  \end{aligned}
	\end{equation}
	Recalling that $g (G) = O (|G|^2)$, we have $g' (G) = O (|G|)$. Since $G = \nabla \psi$, we know that $\| g' (G) \|_{H^s} \leq C \| \nabla \psi \|_{H^s}$. We thereby estimate the term $I\!V_2$ that
	\begin{equation}\label{IV2}
	  \begin{aligned}
	    I\!V_2 = & - \sum_{m' \leq m} C_m^{m'} \big\langle \partial^{m'} g' (G) \partial^{m-m'} \partial_t G , \nabla \partial^m \partial_t v \big\rangle \\
	    \leq & C \sum_{m' \leq m} \| \partial^{m'} g' (G) \|_{L^\infty} \| \nabla \partial^{m-m'} \partial_t \psi \|_{L^2} \| \nabla \partial^m \partial_t v \|_{L^2} \\
	    \leq & C \sum_{m' \leq m} \| \partial^{m'} g' (G) \|_{H^2} \| \nabla \partial^{m-m'} \partial_t \psi \|_{L^2} \| \nabla \partial^m \partial_t v \|_{L^2} \\
	    \leq & C \| g' (G) \|_{H^s} \| \nabla \partial_t \psi \|_{H^{s-2}} \| \nabla \partial_t v \|_{H^{s-2}} \\
	    \leq & C \| \nabla \psi \|_{H^s} \| \nabla \partial_t \psi \|_{H^{s-2}} \| \nabla \partial_t v \|_{H^{s-2}} \\
	    \leq & C (\delta) \mathbb{E}_s^\frac{1}{2} (t) \mathbb{D}_s (t) \,,
	  \end{aligned}
	\end{equation}
	where the Sobolev embedding $H^2 (\mathbb{R}^d) \hookrightarrow L^\infty (\mathbb{R}^d)$ with $d=2,3$ is also used.
	
	It remains to estimate the term $I\!V_3$. One observes that
	\begin{equation}\no
	  \begin{aligned}
	    I\!V_3 = \underset{I\!V_{31}}{\underbrace{ - \big\langle \partial^m ( \nabla \partial_t M \odot \nabla M ) , \nabla \partial^m \partial_t v \big\rangle }} \ \underset{I\!V_{32}}{\underbrace{ -  \big\langle \partial^m ( \nabla M \odot \nabla \partial_t M ) , \nabla \partial^m \partial_t v \big\rangle }}\,.
	  \end{aligned}
	\end{equation}
	By using the last second $M$-equation of \eqref{MEL-reformulate}, the term $I\!V_{31}$ can be rewritten as
	\begin{equation}\label{IV311-IV312-IV313-IV314}
	  \begin{aligned}
	    I\!V_{31} = & \underset{I\!V_{311}}{\underbrace{ \big\langle \partial^m ( \nabla (v \cdot \nabla M) \odot \nabla M ) , \nabla \partial^m \partial_t v \big\rangle }} \  \underset{I\!V_{312}}{\underbrace{ - \big\langle \partial^m ( \nabla (|\nabla M|^2 M) \odot \nabla M ) , \nabla \partial^m \partial_t v \big\rangle }} \\
	    & \underset{I\!V_{313}}{\underbrace{ - \big\langle \partial^m ( \nabla ( \Delta M) \odot \nabla M ) , \nabla \partial^m \partial_t v \big\rangle }} + \underset{I\!V_{314}}{\underbrace{ \big\langle \partial^m ( \nabla (M \times \Delta M) \odot \nabla M ) , \nabla \partial^m \partial_t v \big\rangle }}\,.
	  \end{aligned}
	\end{equation}
	For the term $I\!V_{311}$, we estimate that
	\begin{equation}\label{IV311}
	  \begin{aligned}
	    I\!V_{311} = & \sum_{m' \leq m , m'' \leq m'} C_m^{m'} C_{m'}^{m''} \big\langle ( \nabla \partial^{m''} v \cdot \nabla \partial^{m'-m''} M ) \odot \nabla \partial^{m-m'} M , \nabla \partial^m \partial_t v \big\rangle \\
	    + & \sum_{m' \leq m , m'' \leq m'} C_m^{m'} C_{m'}^{m''} \big\langle ( \partial^{m''} v \cdot \nabla \partial^{m'-m''} \nabla M ) \odot \nabla \partial^{m-m'} M , \nabla \partial^m \partial_t v \big\rangle \\
	    \leq & C \sum_{m' \leq m , m'' \leq m'} \big( \| \nabla \partial^{m''} v \|_{L^2} \| \nabla \partial^{m'-m''} M \|_{L^\infty} + \| \partial^{m''} v \|_{L^\infty} \| \nabla \partial^{m'-m''} \nabla M \|_{L^2} \big) \\
	    & \qquad \qquad \qquad \qquad \times \| \nabla \partial^{m-m'} M \|_{L^\infty} \| \nabla \partial^m \partial_t v \|_{L^2} \\
	    \leq & C \sum_{m' \leq m , m'' \leq m'} \big( \| \nabla \partial^{m''} v \|_{L^2} \| \nabla \partial^{m'-m''} M \|_{H^2} + \| \partial^{m''} v \|_{H^2} \| \nabla \partial^{m'-m''} \nabla M \|_{L^2} \big) \\
	    & \qquad \qquad \qquad \qquad \times \| \nabla \partial^{m-m'} M \|_{H^2} \| \nabla \partial^m \partial_t v \|_{L^2} \\
	    \leq & C ( \| \nabla v \|_{H^s} \| \nabla M \|_{H^s} + \| v \|_{H^s} \| \Delta M \|_{H^s} ) \| \nabla M \|_{H^s} \| \nabla \partial_t v \|_{H^{s-2}} \\
	    \leq & C(\delta) \mathbb{E}_s (t) \mathbb{D}_s (t) \,,
	  \end{aligned}
	\end{equation}
	where the last third inequality is implied by the Sobolev embedding $H^2 (\mathbb{R}^d) \hookrightarrow L^\infty (\mathbb{R}^d)$ with $d = 2,3$. For the term $I\!V_{312}$, we have
	\begin{equation}\label{IV312}
	  \begin{aligned}
	    I\!V_{312} = & - \sum_{m' \leq m , m'' \leq m'} C_m^{m'} C_{m'}^{m''} \big\langle \partial^{m-m'} |\nabla M|^2 \nabla \partial^{m''} M \odot \nabla \partial^{m' - m''} M, \nabla \partial^m \partial_t v \big\rangle \\
	    & - 2 \sum_{m' \leq m} C_m^{m'} \big\langle \nabla \partial^{m'} \nabla M \cdot \partial^{m-m'} ( (\nabla M) M \odot \nabla M ) , \nabla \partial^m \partial_t v \big\rangle \\
	    \leq & C \sum_{m' \leq m , m'' \leq m'} \| \partial^{m-m'} |\nabla M|^2 \|_{L^\infty} \| \nabla \partial^{m''} M \|_{L^4} \| \nabla \partial^{m' - m''} M \|_{L^4} \| \nabla \partial^m \partial_t v \|_{L^2} \\
	    & + C \sum_{m' \leq m} \| \nabla \partial^{m'} \nabla M \|_{L^2} \| \partial^{m-m'} ( (\nabla M) M \odot \nabla M ) \|_{L^\infty} \| \nabla \partial^m \partial_t v \|_{L^2} \\
	    \leq & C \sum_{m' \leq m , m'' \leq m'} \| \partial^{m-m'} |\nabla M|^2 \|_{H^2} \| \nabla \partial^{m''} M \|_{L^2}^{1-\frac{d}{4}} \| \Delta \partial^{m''} M \|_{L^2}^{\frac{d}{4}} \\
	    & \qquad \qquad \qquad \times \| \nabla \partial^{m' - m''} M \|_{L^2}^{1-\frac{d}{4}} \| \Delta \partial^{m' - m''} M \|_{L^2}^{\frac{d}{4}} \| \nabla \partial^m \partial_t v \|_{L^2} \\
	    & + C \sum_{m' \leq m} \| \nabla \partial^{m'} \nabla M \|_{L^2} \| \partial^{m-m'} ( (\nabla M) M \odot \nabla M ) \|_{H^2} \| \nabla \partial^m \partial_t v \|_{L^2} \\
	    \leq & C ( 1 + \| \nabla M \|_{H^s} ) \| \nabla M \|^2_{H^s} \| \Delta M \|_{H^s} \| \nabla \partial_t v \|_{H^{s-2}} \\
	    \leq & C (\delta) \big( 1 + \mathbb{E}_s^\frac{1}{2} (t) \big) \mathbb{E}_s (t) \mathbb{D}_s (t) \,,
	  \end{aligned}
	\end{equation}
	where the geometric constraint $|M| = 1$ and the Sobolev embedding $H^2 (\mathbb{R}^d) \hookrightarrow L^\infty (\mathbb{R}^d)$ $(d=2,3)$ are utilized. For the term $I\!V_{313}$, similarly in \eqref{IV312}, we deduce that
	\begin{equation}\label{IV313}
	  \begin{aligned}
	    I\!V_{313} = & - \sum_{m' \leq m} C_m^{m'} \big\langle \partial^{m'} \nabla \Delta M \odot \nabla \partial^{m-m'} M , \nabla \partial^m \partial_t v \big\rangle \\
	    \leq & C \sum_{m' \leq m} \| \partial^{m'} \nabla \Delta M \|_{L^2} \| \nabla \partial^{m-m'} M \|_{L^\infty} \| \nabla \partial^m \partial_t v \|_{L^2} \\
	    \leq & C \sum_{m' \leq m} \| \partial^{m'} \nabla \Delta M \|_{L^2} \| \nabla \partial^{m-m'} M \|_{H^2} \| \nabla \partial^m \partial_t v \|_{L^2} \\
	    \leq & C \| \nabla M \|_{H^s} \| \Delta M \|_{H^s} \| \nabla \partial_t v \|_{H^{s-2}} \\
	    \leq & C(\delta) \mathbb{E}_s^\frac{1}{2} (t) \mathbb{D}_s (t) \,,
	  \end{aligned}
	\end{equation}
	where the Sobolev embedding $H^2 (\mathbb{R}^d) \hookrightarrow L^\infty (\mathbb{R}^d)$ $(d=2,3)$ is also utilized. For the term $I\!V_{314}$, we estimate that
	\begin{equation}\label{IV314}
	  \begin{aligned}
	    I\!V_{314} = & \sum_{m' \leq m , m'' \leq m'} C_m^{m'} C_{m'}^{m''} \big\langle ( \nabla \partial^{m'-m''} M \times \Delta \partial^{m''} M  ) \odot \nabla \partial^{m-m'} M , \nabla \partial^m \partial_t v \big\rangle \\
	    & + \sum_{m' \leq m , m'' \leq m'} C_m^{m'} C_{m'}^{m''} \big\langle ( \partial^{m'-m''} M \times \nabla \Delta \partial^{m''} M ) \odot \nabla \partial^{m-m'} M , \nabla \partial^m \partial_t v \big\rangle \\
	    \leq & C \sum_{m' \leq m , m'' \leq m'} \| \nabla \partial^{m'-m''} M \|_{L^\infty} \| \Delta \partial^{m''} M \|_{L^2} \| \nabla \partial^{m-m'} M \|_{L^\infty} \| \nabla \partial^m \partial_t v \|_{L^2} \\
	    & + C \sum_{m' \leq m , m'' \leq m'} \| \partial^{m'-m''} M \|_{L^\infty} \| \nabla \Delta \partial^{m''} M \|_{L^2} \| \nabla \partial^{m-m'} M \|_{L^\infty} \| \nabla \partial^m \partial_t v \|_{L^2} \\
	    \leq & C \big[ \| \nabla M \|_{H^s} \| \Delta M \|_{H^s} + ( 1 + \| \nabla M \|_{H^s} ) \| \Delta M \|_{H^s} \big] \| \nabla M \|_{H^s} \| \nabla \partial_t v \|_{H^{s-2}} \\
	    \leq & C ( 1 + \| \nabla M \|_{H^s} ) \| \nabla M \|_{H^s} \| \Delta M \|_{H^s} \| \nabla \partial_t v \|_{H^{s-2}} \\
	    \leq & C(\delta) \big( \mathbb{E}_s^\frac{1}{2} (t) + \mathbb{E}_s (t) \big) \mathbb{D}_s (t) \,.
	  \end{aligned}
	\end{equation}
	Consequently, we plug the bounds \eqref{IV311}, \eqref{IV312}, \eqref{IV313} and \eqref{IV314} into the equality \eqref{IV311-IV312-IV313-IV314} and then we obtain
	\begin{equation}\label{IV31}
	  \begin{aligned}
	    I\!V_{31} \leq C(\delta) \big( \mathbb{E}_s^\frac{1}{2} (t) + \mathbb{E}_s^\frac{3}{2} (t) \big) \mathbb{D}_s (t) \,.
	  \end{aligned}
	\end{equation}
	From employing the similar arguments in \eqref{IV31}, we can deduce that
	\begin{equation}\no
	  \begin{aligned}
	    I\!V_{32} \leq C(\delta) \big( \mathbb{E}_s^\frac{1}{2} (t) + \mathbb{E}_s^\frac{3}{2} (t) \big) \mathbb{D}_s (t) \,.
	  \end{aligned}
	\end{equation}
	We thereby have
	\begin{equation}\label{IV3}
	  \begin{aligned}
	    I\!V_3 = I\!V_{31} + I\!V_{32} \leq C(\delta) \big( \mathbb{E}_s^\frac{1}{2} (t) + \mathbb{E}_s^\frac{3}{2} (t) \big) \mathbb{D}_s (t) \,.
	  \end{aligned}
	\end{equation}
	As a result, by substituting the bounds \eqref{IV1}, \eqref{IV2} and \eqref{IV3} into the equality \eqref{IV1-IV2-IV3} and summing up for $|m| \leq s - 2$, we imply that
	\begin{equation}\label{norm-vt-Hs-2}
	  \begin{aligned}
	    \tfrac{1}{2} \tfrac{\d}{\d t} \| \partial_t v \|^2_{H^{s-2}} + \nu \| \nabla \partial_t v \|^2_{H^{s-2}} + \sum_{|m| \leq s - 2} \big\langle \Delta \partial^m \partial_t \psi , \partial^m \partial_t v \big\rangle \\
	    \leq C(\delta) \big( \mathbb{E}_s^\frac{1}{2} (t) + \mathbb{E}_s^\frac{3}{2} (t) \big) \mathbb{D}_s (t) \,.
	  \end{aligned}
	\end{equation}
	
	{\em Step 5. Estimate the norm $\| \nabla \partial_t \psi \|^2_{H^{s-2}}$.} For all $|m| \leq s - 2$, we act the derivative operator $\partial^m \partial_t$ on the third $\psi$-equation and take $L^2$-inner product by multiplying $\Delta \partial^m \partial_t \psi$. We then obtain
	\begin{equation}\label{V-equ}
	  \begin{aligned}
	    \tfrac{1}{2} \tfrac{\d}{\d t} \| \nabla \partial^m \partial_t \psi \|^2_{L^2} - \big\langle \Delta \partial^m \partial_t \psi , \partial^m \partial_t v \big\rangle = \underset{V}{\underbrace{ - \big\langle \partial^m \nabla \partial_t ( v \cdot \nabla \psi ) , \nabla \partial^m \partial_t \psi \big\rangle }} \,.
	  \end{aligned}
	\end{equation}
	We first decompose the term $V$ into four parts:
	\begin{equation}\label{V-V1-V4}
	  \begin{aligned}
	    V = \ & \underset{V_1}{\underbrace{ - \big\langle \partial^m ( \nabla \partial_t v \cdot \nabla \psi ) , \nabla \partial^m \partial_t \psi \big\rangle }} \ \ \underset{V_2}{\underbrace{ - \big\langle \partial^m ( \partial_t v \cdot \nabla \nabla \psi ) , \nabla \partial^m \partial_t \psi \big\rangle }} \\
	    & \underset{V_3}{\underbrace{ - \big\langle \partial^m ( \nabla v \cdot \nabla \partial_t \psi ) , \nabla \partial^m \partial_t \psi \big\rangle }} \ \ \underset{V_4}{\underbrace{ - \big\langle \partial^m ( v \cdot \nabla \nabla \partial_t \psi ) , \nabla \partial^m \partial_t \psi \big\rangle }} \,.
	  \end{aligned}
	\end{equation}
	It is derived from the Sobolev embedding $H^2 (\mathbb{R}^d) \hookrightarrow L^\infty (\mathbb{R}^d)$ for $d = 2, 3$ that
	\begin{equation}\label{V1}
	  \begin{aligned}
	    V_1 = & - \sum_{m' \leq m} C_m^{m'} \big\langle \nabla \partial^{m-m'} \partial_t v \cdot \nabla \partial^{m'} \psi , \nabla \partial^m \partial_t \psi \big\rangle \\
	    \leq & C \sum_{m' \leq m} \| \nabla \partial^{m-m'} \partial_t v \|_{L^2} \| \nabla \partial^{m'} \psi \|_{L^\infty} \| \nabla \partial^m \partial_t \psi \|_{L^2} \\
	    \leq & C \sum_{m' \leq m} \| \nabla \partial^{m-m'} \partial_t v \|_{L^2} \| \nabla \partial^{m'} \psi \|_{H^2} \| \nabla \partial^m \partial_t \psi \|_{L^2} \\
	    \leq & C \| \nabla \psi \|_{H^s} \| \nabla \partial_t \psi \|_{H^{s-2}} \| \nabla \partial_t v \|_{H^{s-2}} \\
	    \leq & C (\delta) \mathbb{E}_s^\frac{1}{2} (t) \mathbb{D}_s (t) \,.
	  \end{aligned}
	\end{equation}
	From the Sobolev interpolation inequality $\| f \|_{L^4 (\mathbb{R}^d)} \leq C \| f \|_{L^2 (\mathbb{R}^d)}^{1-\frac{d}{4}} \| \nabla f \|_{L^2 (\mathbb{R}^d)}^\frac{d}{4} $ and Sobolev embedding $H^1 (\mathbb{R}^d) \hookrightarrow L^4 (\mathbb{R}^d)$ with $d=2,3$, we deduce that
	\begin{equation}\label{V2}
	  \begin{aligned}
	    V_2 = & - \sum_{m' \leq m} C_m^{m'} \big\langle \partial^{m'} \partial_t v \cdot \nabla \partial^{m-m'} \nabla \psi , \nabla  \partial^m \partial_t \psi \big\rangle \\
	    \leq & C \sum_{m' \leq m} \| \partial^{m'} \partial_t v \|_{L^4} \| \nabla \partial^{m-m'} \nabla \psi \|_{L^4} \| \nabla  \partial^m \partial_t \psi \|_{L^2} \\
	    \leq & C \sum_{m' \leq m} \| \partial^{m'} \partial_t v \|_{L^2}^{1-\frac{d}{4}} \| \nabla \partial^{m'} \partial_t v \|_{L^2}^\frac{d}{4} \| \nabla \partial^{m-m'} \nabla \psi \|_{H^1} \| \nabla  \partial^m \partial_t \psi \|_{L^2} \\
	    \leq & C \| \partial_t v \|_{H^{s-2}}^{1-\frac{d}{4}} \| \nabla \psi \|_{H^s}^\frac{d}{4} \| \nabla \psi \|_{H^s}^{1-\frac{d}{4}} \| \nabla \partial_t v \|_{H^{s-2}}^\frac{d}{4} \| \nabla \partial_t \psi \|_{H^{s-2}} \\
	    \leq & C (\delta) \mathbb{E}_s^\frac{1}{2} (t) \mathbb{D}_s (t) \,.
	  \end{aligned}
	\end{equation}
	For the term $I\!V_3$, we estimate that
	\begin{equation}\label{V3}
	  \begin{aligned}
	    V_3 = & - \sum_{m' \leq m} C_m^{m'} \big\langle \nabla \partial^{m'} v \cdot \nabla \partial^{m-m'} \partial_t \psi , \nabla \partial^m \partial_t \psi \big\rangle \\
	    \leq & C \sum_{m' \leq m} \| \nabla \partial^{m'} v \|_{L^\infty} \| \nabla \partial^{m-m'} \partial_t \psi \|_{L^2} \| \nabla \partial^m \partial_t \psi \|_{L^2} \\
	    \leq & C \sum_{m' \leq m} \| \nabla \partial^{m'} v \|_{H^2} \| \nabla \partial^{m-m'} \partial_t \psi \|_{L^2} \| \nabla \partial^m \partial_t \psi \|_{L^2} \\
	    \leq & C \| \nabla v \|_{H^s} \| \nabla \partial \psi \|^2_{H^{s-2}} \leq C (\delta) \mathbb{E}_s^\frac{1}{2} (t) \mathbb{D}_s (t) \,,
	  \end{aligned}
	\end{equation}
	where the Sobolev embedding $H^2 (\mathbb{R}^d) \hookrightarrow L^\infty (\mathbb{R}^d)$ for $d = 2, 3$ is used. Thanks to the divergence-free property of the velocity field $v$, the term $V_4$ can be estimated as
	\begin{equation}\label{V4}
	  \begin{aligned}
	    V_4 = & - \sum_{0 \neq m' \leq m} C_m^{m'} \big\langle \partial^{m'} v \cdot \nabla \nabla \partial^{m-m'} \partial_t \psi , \nabla \partial^m \partial_t \psi \big\rangle \\
	    \leq & C \sum_{0 \neq m' \leq m} \| \partial^{m'} v \|_{L^\infty} \| \nabla \nabla \partial^{m-m'} \partial_t \psi \|_{L^2} \| \nabla \partial^m \partial_t \psi \|_{L^2} \\
	    \leq & C \sum_{0 \neq m' \leq m} \| \partial^{m'} v \|_{H^2} \| \nabla \nabla \partial^{m-m'} \partial_t \psi \|_{L^2} \| \nabla \partial^m \partial_t \psi \|_{L^2} \\
	    \leq & C \| v \|_{H^s} \| \nabla \partial_t \psi \|^2_{H^{s-2}} \leq C (\delta) \mathbb{E}_s^\frac{1}{2} (t) \mathbb{D}_s (t) \,.
	  \end{aligned}
	\end{equation}
	Plugging the estimates \eqref{V1}, \eqref{V2}, \eqref{V3}, \eqref{V4} into \eqref{V-V1-V4}, we obtain
	\begin{equation}\no
	  \begin{aligned}
	    V \leq C (\delta) \mathbb{E}_s^\frac{1}{2} (t) \mathbb{D}_s (t) \,,
	  \end{aligned}
	\end{equation}
	which immediately implies by substituting the previous bound of $V$ into \eqref{V-equ} and summing up for all $|m| \leq s-2$ that
	\begin{equation}\label{norm-nabla-psi-t-Hs-2-1}
	  \begin{aligned}
	    \tfrac{1}{2} \tfrac{\d}{\d t} \| \nabla \partial_t \psi \|^2_{H^{s-2}} - \sum_{|m| \leq s - 2} \big\langle \Delta \partial^m \partial_t \psi , \partial^m \partial_t v \big\rangle \leq C (\delta) \mathbb{E}_s^\frac{1}{2} (t) \mathbb{D}_s (t) \,.
	  \end{aligned}
	\end{equation}
	Furthermore, we take derivative operator $\nabla$ on the third $\psi$-equation of \eqref{MEL-reformulate}, namely,
	\begin{equation}\no
	  \begin{aligned}
	    \nabla \partial_t \psi = - \nabla v - \nabla ( v \cdot \nabla \psi ) \,.
	  \end{aligned}
	\end{equation}
	We thereby deduce that
	\begin{equation*}
	  \begin{aligned}
	    \| \nabla \partial_t \psi \|^2_{H^{s-2}} \leq & \| \nabla v \|^2_{H^{s-2}} + \| \nabla ( v \cdot \nabla \psi ) \|^2_{H^{s-2}} \\
	    \leq & \| \nabla v \|^2_{H^s} + \| v \cdot \nabla \psi \|^2_{H^{s-1}} \\
	    \leq & \| \nabla v \|^2_{H^s} + C \| v \|^2_{H^s} \| \nabla \psi \|^2_{H^s} \\
	    \leq & \| \nabla v \|^2_{H^s} + C (\delta) \mathbb{E}_s (t) \mathbb{D}_s (t) \,.
	  \end{aligned}
	\end{equation*}
	Hence,
	\begin{equation}\label{norm-nabla-psi-t-Hs-2-2}
	  \begin{aligned}
	     \| \nabla \partial_t \psi \|^2_{H^{s-2}} - \| \nabla v \|^2_{H^s} \leq C (\delta) \mathbb{E}_s (t) \mathbb{D}_s (t) \,.
	  \end{aligned}
	\end{equation}
	
	{\em Step 6. Close the energy estimates.} We first add the estimate \eqref{norm-vt-Hs-2} to \eqref{norm-nabla-psi-t-Hs-2-1}, so that the unsigned quantity $\sum_{|m| \leq s-2} \big\langle \Delta \partial^m \partial_t \psi , \partial^m \partial_t v \big\rangle $ will be eliminated. More precisely, we obtain
	\begin{equation}\label{Close-1}
	  \begin{aligned}
	    \tfrac{1}{2} \tfrac{\d}{\d t} \big( \| \partial_t v \|^2_{H^{s-2}} + \| \nabla \partial_t \psi \|^2_{H^{s-2}} \big) + \nu \| \nabla \partial_t v \|^2_{H^{s-2}} \leq C(\delta) \big( \mathbb{E}_s^\frac{1}{2} (t) + \mathbb{E}_s^\frac{3}{2} (t) \big) \mathbb{D}_s (t) \,.
	  \end{aligned}
	\end{equation}
	In order to absorb the negative term $- \tfrac{4 c_0 K_s}{\nu} \| \nabla \partial_t v \|_{H^{s-2}} \| \nabla \psi \|_{H^s}$ in the inequality \eqref{norm-nabla-psi-Hs}, we multiply \eqref{norm-nabla-psi-Hs} by a small $\delta > 0$ to be determined and add it to the previous inequality \eqref{Close-1}. We then have
	\begin{equation}\label{Close-2}
	  \begin{aligned}
	    & \tfrac{1}{2} \tfrac{\d}{\d t} \big( \delta \| \nabla \psi \|^2_{H^s} + \| \partial_t v \|^2_{H^{s-2}} + \| \nabla \partial_t \psi \|^2_{H^{s-2}} \big) \\
	    + & \tfrac{\delta}{\nu} \| \nabla \psi \|^2_{H^s} + \nu \| \nabla \partial_t v \|^2_{H^{s-2}} - \tfrac{4 c_0 K_s}{\nu} \delta \| \nabla \partial_t v \|_{H^{s-2}} \| \nabla \psi \|_{H^s} \\
	    \leq & C(\delta) \big( \mathbb{E}_s^\frac{1}{2} (t) + \mathbb{E}_s^\frac{3}{2} (t) \big) \mathbb{D}_s (t) \,.
	  \end{aligned}
	\end{equation}
	We then will absorb the negative term $- \| \nabla v \|_{H^s} \| \nabla \psi \|_{H^s}$ in the inequality \eqref{norm-v-Hs} by adding $\delta^2$ times of \eqref{norm-v-Hs} to the previous bound \eqref{Close-2}. To be more precise, we obtain
	\begin{equation}\label{Close-3}
	  \begin{aligned}
	    & \tfrac{1}{2} \tfrac{\d}{\d t} \big( \delta^2 \| v \|^2_{H^s} + \delta \| \nabla \psi \|^2_{H^s} + \| \partial_t v \|^2_{H^{s-2}} + \| \nabla \partial_t \psi \|^2_{H^{s-2}} \big) \\
	    + & \delta^2 \nu \| \nabla v \|^2_{H^s} + \tfrac{\delta}{\nu} \| \nabla \psi \|^2_{H^s} + \nu \| \nabla \partial_t v \|^2_{H^{s-2}} \\
	    - & \tfrac{4 c_0 K_s}{\nu} \delta \| \nabla \partial_t v \|_{H^{s-2}} \| \nabla \psi \|_{H^s} - \delta^2 \| \nabla v \|_{H^s} \| \nabla \psi \|_{H^s} \\
	    \leq & C(\delta) \big( \mathbb{E}_s^\frac{1}{2} (t) + \mathbb{E}_s^\frac{3}{2} (t) \big) \mathbb{D}_s (t) \,.
	  \end{aligned}
	\end{equation}
	Furthermore, we multiply \eqref{norm-nabla-psi-t-Hs-2-2} by $\tfrac{1}{2} \delta^2 \nu$ and add it to \eqref{Close-3}, so that the negative term $- \| \nabla v \|^2_{H^s}$ in \eqref{norm-nabla-psi-t-Hs-2-2} will be absorbed by choosing a proper small $\delta > 0$. More precisely, we have
	\begin{equation}
	  \begin{aligned}
	    & \tfrac{1}{2} \tfrac{\d}{\d t} \big( \delta^2 \| v \|^2_{H^s} + \delta \| \nabla \psi \|^2_{H^s} + \| \partial_t v \|^2_{H^{s-2}} + \| \nabla \partial_t \psi \|^2_{H^{s-2}} \big) \\
	    + & \tfrac{1}{2} \delta^2 \nu \| \nabla v \|^2_{H^s} + \tfrac{\delta}{\nu} \| \nabla \psi \|^2_{H^s} + \nu \| \nabla \partial_t v \|^2_{H^{s-2}} + \tfrac{1}{2} \delta^2 \nu \| \nabla \partial_t \psi \|^2_{H^{s-2}} \\
	    - & \tfrac{4 c_0 K_s}{\nu} \delta \| \nabla \partial_t v \|_{H^{s-2}} \| \nabla \psi \|_{H^s} - \delta^2 \| \nabla v \|_{H^s} \| \nabla \psi \|_{H^s} \\
	    \leq & C(\delta) \big( \mathbb{E}_s^\frac{1}{2} (t) + \mathbb{E}_s^\frac{3}{2} (t) \big) \mathbb{D}_s (t) \,.
	  \end{aligned}
	\end{equation}
	Finally, we add \eqref{norm-nabla-M-Hs} to the above inequality and obtain
	\begin{equation}\no
	  \begin{aligned}
	    & \tfrac{1}{2} \tfrac{\d}{\d t} \big( \delta^2 \| v \|^2_{H^s} + \delta \| \nabla \psi \|^2_{H^s} + \| \nabla M \|^2_{H^s} + \| \partial_t v \|^2_{H^{s-2}} + \| \nabla \partial_t \psi \|^2_{H^{s-2}} \big) \\
	    + & \tfrac{1}{2} \delta^2 \nu \| \nabla v \|^2_{H^s} + \tfrac{\delta}{\nu} \| \nabla \psi \|^2_{H^s} + \| \Delta M \|^2_{H^s} + \nu \| \nabla \partial_t v \|^2_{H^{s-2}} + \tfrac{1}{2} \delta^2 \nu \| \nabla \partial_t \psi \|^2_{H^{s-2}} \\
	    - & \tfrac{4 c_0 K_s}{\nu} \delta \| \nabla \partial_t v \|_{H^{s-2}} \| \nabla \psi \|_{H^s} - \delta^2 \| \nabla v \|_{H^s} \| \nabla \psi \|_{H^s} \\
	    \leq & C(\delta) \big( \mathbb{E}_s^\frac{1}{2} (t) + \mathbb{E}_s^\frac{3}{2} (t) \big) \mathbb{D}_s (t) \,.
	  \end{aligned}
	\end{equation}
	Notice that the Young's inequality reduces to
	\begin{equation*}
	  \begin{aligned}
	    & - \tfrac{4 c_0 K_s}{\nu} \delta \| \nabla \partial_t v \|_{H^{s-2}} \| \nabla \psi \|_{H^s} \geq - \tfrac{\delta}{2 \nu} \| \nabla \psi \|^2_{H^s} - \tfrac{16 c_0^2 K_s^2}{2 \nu} \delta \| \nabla \partial_t v \|^2_{H^{s-2}} \,, \\
	    & - \delta^2 \| \nabla v \|_{H^s} \| \nabla \psi \|_{H^s} \geq - \nu \delta^3 \| \nabla v \|^2_{H^s} - \tfrac{\delta}{4 \nu} \| \nabla \psi \|^2_{H^s} \,,
	  \end{aligned}
	\end{equation*}
	which implies that
	\begin{equation}
	  \begin{aligned}
	    & \tfrac{1}{2} \delta^2 \nu \| \nabla v \|^2_{H^s} + \tfrac{\delta}{\nu} \| \nabla \psi \|^2_{H^s} + \| \Delta M \|^2_{H^s} + \nu \| \nabla \partial_t v \|^2_{H^{s-2}} + \tfrac{1}{2} \delta^2 \nu \| \nabla \partial_t \psi \|^2_{H^{s-2}} \\
	    & - \tfrac{4 c_0 K_s}{\nu} \delta \| \nabla \partial_t v \|_{H^{s-2}} \| \nabla \psi \|_{H^s} - \delta^2 \| \nabla v \|_{H^s} \| \nabla \psi \|_{H^s} \\
	    \geq & \tfrac{1}{4} \nu \delta^2 \| \nabla v \|^2_{H^s} + \tfrac{\delta}{4 \nu} \| \nabla \psi \|^2_{H^s} + \| \Delta M \|^2_{H^s} + \tfrac{1}{2} \nu \| \nabla \partial_t v \|^2_{H^{s-2}} + \tfrac{1}{2} \nu \delta^2 \| \nabla \partial_t \psi \|^2_{H^{s-2}} \\
	    = & \tfrac{1}{2} \mathbb{D}_s (t)
	  \end{aligned}
	\end{equation}
	provided that $\tfrac{1}{2} \delta^2 \nu - \nu \delta^3 \geq \tfrac{1}{4} \delta^2 \nu > 0$ and $\nu - \tfrac{16 c_0^2 K_s^2}{2 \nu} \delta \geq \tfrac{1}{2} \nu$, i.e.,
	\begin{equation}\label{Small-delta}
	  \begin{aligned}
	    0 < \delta \leq \min \big\{ \tfrac{1}{4} , \tfrac{\nu^2}{16 c_0^2 K_s^2} \big\} \,.
	  \end{aligned}
	\end{equation}
	Consequently, we have
	\begin{equation*}
	  \begin{aligned}
	    \tfrac{\d}{\d t} \mathbb{E}_s (t) + \mathbb{D}_s (t) \leq C(\delta) \big( \mathbb{E}_s^\frac{1}{2} (t) + \mathbb{E}_s^\frac{3}{2} (t) \big) \mathbb{D}_s (t) \,,
	  \end{aligned}
	\end{equation*}
	and the proof of Proposition \ref{Prop-Global-Est} is finished.	
\end{proof}

\noindent{\bf Proof of Theorem \ref{Thm-Global}: global well-posedness with small initial data.} Based on Proposition \ref{Prop-Global-Est}, we now start to prove the global-in-time existence to the system \eqref{MEL-reformulate}, which is equivalent to the equations \eqref{MEL} with given $H_{ext} = 0$. First, by using the system \eqref{MEL-original} and $\psi (0)$ obeying $\nabla \psi (0) = G_0$, we know that
\begin{equation}
  \begin{aligned}
    & \| \partial_t v (0) \|_{H^{s-2}} \leq \nu \| \Delta v_0 \|_{H^{s-2}} + \| \mathcal{P} \nabla \cdot G_0 \|_{H^{s-2}} + \| \mathcal{P} (v_0 \cdot \nabla v_0) \|_{H^{s-2}} \\
    & \qquad \qquad \qquad \quad + \| \mathcal{P} \nabla \cdot g (G_0) \|_{H^{s-2}} + \| \mathcal{P} \nabla \cdot ( \nabla M_0 \odot \nabla M_0 ) \|_{H^{s-2}} \\
    \leq & C ( \| v_0 \|_{H^s} + \| G_0 \|_{H^s} + \| v_0 \|^2_{H^s} + \| G_0 \|^2_{H^s} + \| \nabla M_0 \|^2_{H^s} ) \\
    \leq & C_1 ( 1 + \| v_0 \|_{H^s} + \| F_0 - I \|_{H^s} + \| \nabla M_0 \|_{H^s} ) ( \| v_0 \|_{H^s} + \| F_0 - I \|_{H^s} + \| \nabla M_0 \|_{H^s} ) \,,
  \end{aligned}
\end{equation}
where the symbol $\mathcal{P}$ is the Leray projection on $\R^d$, and
\begin{equation}
  \begin{aligned}
    \| \partial_t \nabla \psi (0) \|_{H^{s-2}} \leq & \| \nabla v_0 \|_{H^{s-2}} + \| v_0 \cdot G_0 \|_{H^{s-2}} \\
    \leq & C ( \| \nabla v_0 \|_{H^{s-2}} + \| \nabla ( v_0 \cdot G_0 ) \|_{H^{s-2}} ) \\
    \leq & C_2  \| v_0 \|_{H^s} ( 1 + \| F_0 - I \|_{H^s} ) \,.
  \end{aligned}
\end{equation}
Consequently, by the definition of energy functional $\mathbb{E}_s (t)$, the inequality \eqref{G-F-relation} and the small size of initial condition \eqref{IC-small-size}, we imply that
\begin{equation}
  \begin{aligned}
    \mathbb{E}_s (0) \leq & ( 4 C_1^2 + C_2^2 + \delta^2 ) ( 1 + \| v_0 \|^2_{H^s} + \| \nabla M_0 \|^2_{H^s} + \| F_0 - I \|^2_{H^s} ) \\
    & \qquad \times ( \| v_0 \|^2_{H^s} + \| \nabla M_0 \|^2_{H^s} + \| F_0 - I \|^2_{H^s} ) \\
    \leq & C_{\!_\#} ( 1 + \eps_0 ) \eps_0 \,,
  \end{aligned}
\end{equation}
where $ C_{\!_\#} = 4 C_1^2 + C_2^2 + \delta^2 > 0 $ and $\eps_0 > 0$ is small to be determined.

We now choose small $\eps_0 \in (0,1)$ such that
\begin{equation*}
  \begin{aligned}
    c_1 \big( \mathbb{E}_s^\frac{1}{2} (0) + \mathbb{E}_s^\frac{3}{2} (0) \big) \leq c_1 \Big( \sqrt{C_{\!_\#} (1 + \eps_0) \eps_0} + \sqrt{C_{\!_\#} (1 + \eps_0)^3 \eps_0} \ \Big) \leq 3 c_1 \sqrt{2 C_{\!_\#}} \sqrt{\eps_0} \leq \tfrac{1}{4} \,,
  \end{aligned}
\end{equation*}
where $c_1 > 0$ is mentioned in Proposition \ref{Prop-Global-Est}. More precisely, if
\begin{equation}\label{Small-eps0}
  \begin{aligned}
    0 < \eps_0 \leq \min \Big\{ 1, \tfrac{1}{288 c_1^2 C_{\!_\#}} \Big\} \,,
  \end{aligned}
\end{equation}
we have
\begin{equation}\label{Small-Initial-Energy}
  \begin{aligned}
    c_1 \big( \mathbb{E}_s^\frac{1}{2} (0) + \mathbb{E}_s^\frac{3}{2} (0) \big) \leq \tfrac{1}{4}
  \end{aligned}
\end{equation}
holds under the small initial conditions \eqref{IC-small-size} with small $\eps_0 > 0$ given in \eqref{Small-eps0}. We define the following number
\begin{equation}
  \begin{aligned}
    T^* = \sup \Big\{ \tau > 0 ; \sup_{t \in [0, \tau]} c_1 \big( \mathbb{E}_s^\frac{1}{2} (t) + \mathbb{E}_s^\frac{3}{2} (t) \big) \leq \tfrac{1}{2} \Big\} \geq 0 \,.
  \end{aligned}
\end{equation}
Then, from the initial bound \eqref{Small-Initial-Energy} and the continuity of the energy functional $\mathbb{E}_s (t)$, we deduce that $T^* > 0$. Thus for all $t \in [0, T^*]$
\begin{equation*}
  \begin{aligned}
    \tfrac{\d}{\d t} \mathbb{E}_s (t) + \mathbb{D}_s (t) \leq 0 \,,
  \end{aligned}
\end{equation*}
which immediately means that $ \mathbb{E}_s (t) \leq \mathbb{E}_s (0) \leq C_{\!_\#} ( 1 + \eps_0 ) \eps_0 $ holds for all $t \in [0, T^*]$. Consequently,
\begin{equation}
  \begin{aligned}
    \sup_{t \in [0, T^*]} c_1 \big( \mathbb{E}_s^\frac{1}{2} (t) + \mathbb{E}_s^\frac{3}{2} (t) \big) \leq \tfrac{1}{4} \,.
  \end{aligned}
\end{equation}
We now claim that $T^* = + \infty$. Otherwise, if $T^* < + \infty$, the continuity of the energy $\mathbb{E}_s (t)$ implies that there is a small positive $\eps > 0$ such that
\begin{equation*}
  \begin{aligned}
    \sup_{t \in [0, T^* + \eps]} c_1 \big( \mathbb{E}_s^\frac{1}{2} (t) + \mathbb{E}_s^\frac{3}{2} (t) \big) \leq \tfrac{3}{8} < \tfrac{1}{2} \,,
  \end{aligned}
\end{equation*}
which contradicts to the definition of $T^*$. Therefore, we get
\begin{equation}
  \begin{aligned}
    \mathbb{E}_s (t) + \int_0^t \mathbb{D}_s (\tau) \d \tau \leq \mathbb{E}_s (t) \leq 2 C_{\!_\#} \eps_0
  \end{aligned}
\end{equation}
for all $t \in \R^+$, and as a consequence, the proof of Theorem \ref{Thm-Global} is finished.

\section*{Acknowledgment}

This work was supported by grants from the National Natural Science Foundation of China under contract No. 11471181 and No. 11731008.

\bibliography{reference}

\end{document}